\documentclass[12pt,reqno]{amsart}
\usepackage{amssymb}
\usepackage{diagrams}
\usepackage{graphicx}
\usepackage{amsmath}
\usepackage{a4wide}
\usepackage{version}
\usepackage{mathrsfs}
\usepackage{tikz}
\usepackage{mathtools}
\usetikzlibrary{arrows,decorations.markings, matrix}
\usepackage{hyperref}

\usepackage[all]{xy}

\numberwithin{equation}{section}

\newtheorem{thm}{Theorem}[subsection]
\newtheorem{lem}[thm]{Lemma}
\newtheorem{prop}[thm]{Proposition}
\newtheorem{cor}[thm]{Corollary}

\theoremstyle{definition}
\newtheorem{definition}[thm]{Definition}
\newtheorem{example}[thm]{Example}
\newtheorem{rmk}[thm]{Remark}

\newcommand{\Z}{\mathbb{Z}}

\newcommand{\obs}{\operatorname{obs}}

\newcommand{\Pf}{\noindent {\it Proof}}
\newcommand{\id}{\operatorname{id}}
\newcommand{\Lie}{\operatorname{Lie}}
\newcommand{\ov}{\overline}
\newcommand{\we}{\wedge}

\newcommand{\rk}{\operatorname{rk}}

\newcommand{\WW}{{\mathcal W}}
\newcommand{\EE}{{\mathcal E}}
\newcommand{\MM}{{\mathcal M}}
\newcommand{\TT}{{\mathcal T}}

\newcommand{\HH}{{\mathcal H}}

\newcommand{\VV}{{\mathcal V}}

\newcommand{\LL}{{\mathcal L}}

\newcommand{\Om}{\Omega}

\newcommand{\Hom}{\operatorname{Hom}}

\newcommand{\Ext}{\operatorname{Ext}}

\newcommand{\Aut}{\operatorname{Aut}}

\renewcommand{\a}{\alpha}

\newcommand{\om}{\omega}

\newcommand{\la}{\lambda}

\newcommand{\wt}{\widetilde}
\newcommand{\ot}{\otimes}

\newcommand{\sub}{\subset}
\newcommand{\ed}{\qed\vspace{3mm}}

\newcommand{\PGL}{\operatorname{PGL}}

\newcommand{\coarse}{\operatorname{coarse}}

\newcommand{\fm}{{\frak m}}

\newcommand{\gl}{{\frak gl}}

\renewcommand{\mod}{\operatorname{mod}}

\newcommand{\Tor}{\operatorname{Tor}}
\newcommand{\und}{\underline}

\newcommand{\OO}{{\mathcal O}}
\newcommand{\RR}{{\mathcal R}}

\newcommand{\Bl}{\operatorname{Bl}}

\newcommand{\coker}{\operatorname{coker}}
\newcommand{\DD}{{\mathcal D}}

\newcommand{\KK}{{\mathcal K}}

\newcommand{\II}{{\mathcal I}}

\newcommand{\SL}{\operatorname{SL}}
\newcommand{\GL}{\operatorname{GL}}

\newcommand{\G}{{\mathbb G}}

\newcommand{\hra}{\hookrightarrow}
\newcommand{\lan}{\langle}
\newcommand{\ran}{\rangle}

\newcommand{\CC}{{\mathcal C}}

\newcommand{\Spec}{\operatorname{Spec}}

\newcommand{\Sp}{\operatorname{Sp}}

\renewcommand{\P}{{\mathbb P}}

\newcommand{\si}{\sigma}

\newcommand{\de}{\delta}
\newcommand{\eps}{\epsilon}

\renewcommand{\ker}{\operatorname{ker}}
\newcommand{\im}{\operatorname{im}}

\newcommand{\A}{{\mathbb A}}

\newcommand{\Fitt}{\operatorname{Fitt}}
\newcommand{\can}{\operatorname{can}}

\newcommand{\tot}{\operatorname{tot}}

\newcommand{\sO}{{\mathcal O}}
\newcommand{\sHom}{\underline{\Hom}}
\newcommand{\red}{\operatorname{red}}
\newcommand{\nest}{\operatorname{nest}}
\newcommand{\Hyp}{{\mathcal H}yp}

\newarrow{Dashto}{}{dash}{}{dash}>

\title{Hyperelliptic limits of quadrics through canonical curves and ribbons}
\author{Alexander Polishchuk}
\address{ 
    University of Oregon; National Research University Higher School of Economics; and Korea Institute for 
    Advanced Study 
  }
  \email{apolish@uoregon.edu}
\author{Eric Rains}
\address{California Institute of Technology}
\email{rains@caltech.edu}
\dedicatory{To David Eisenbud, an offering in ribbons}

\begin{document}

\begin{abstract}
We describe explicitly all hyperelliptic limits of quadrics through smooth canonical curves of genus $g$ in $\P^{g-1}$.
Also, we construct an open embedding of the blow up of a $\PGL_g$-bundle over the moduli space of curves
of genus $g$ along the hyperelliptic locus into the blow up of the canonical Hilbert scheme of $\P^{g-1}$ along the closure
of the locus of canonical ribbons, which are certain double thickenings of rational normal curves introduced and studied
by Bayer and Eisenbud in \cite{BE}.
\end{abstract}

\maketitle

\section*{Introduction}

This paper originated from the following question. Given a smooth
hyperelliptic curve $C_0$ of genus $g\ge 3$, which quadratic relations
between regular differentials can be deformed away from the hyperelliptic
locus?  In other words, the question is which elements of the kernel of the
multiplication map
$$\mu_{C_0}:S^2H^0(C_0,\om_{C_0})\to H^0(C_0,\om_{C_0}^{\ot 2})$$
can be extended along a deformation of $C_0$ with nonhyperelliptic generic
fiber.

An answer to the above question is the first main result of this paper (see
Theorem \ref{quadric-hyperell-thm}).  We show that there is a natural
identification of $\ker(\mu_{C_0})$ with the space of quadratic forms on a
$(g-2)$-dimensional space, and an element in $\ker(\mu_{C_0})$ extends away
from the hyperelliptic locus if and only if the corresponding quadratic
form is degenerate. The most essential part of the proof is the use of {\it
  ribbons} in their canonical embedding.

Recall that ribbons were introduced by Bayer and Eisenbud in \cite{BE}:
these are nonreduced curves that are double thickenings of $\P^1$. We
review the geometry of ribbons in Section \ref{ribbons-sec}.
Nonhyperelliptic ribbons of genus $g$ have canonical embeddings into
$\P^{g-1}$, just like smooth curves, and the original motivation of
\cite{BE} was that their syzygies should be easier to study, and in this
way one can approach Green's Conjecture for generic curves. Note that although Green's Conjecture
for generic curves was resolved by Voisin in a different way, a new recent proof in \cite{AFPRW}
uses degenerations of smooth canonical curves (but not to ribbons).

The main observation that connects ribbons to our question is that if we take a deformation of a smooth hyperelliptic
curve with nonhyperelliptic generic fiber then the limit of the corresponding family of canonical curves in the Hilbert scheme
of $\P^{g-1}$ is a ribbon (canonically embedded). Thus, our question can be solved using some deformation theory together with the description
of quadratic relations for ribbons in the canonical embedding. We also generalize this result to higher order canonical
relations (see Theorem \ref{rel-hyperell-thm}).

The above relation between hyperelliptic curves and ribbons goes back to the paper of Fong \cite{Fong}, where it was
observed that every ribbon arises from some deformation of a fixed hyperelliptic curve
in nonhyperelliptic direction. Our second main result is concerned with unraveling further the connection
between the hyperelliptic locus in $\MM_g$ and the ribbon locus in the canonical Hilbert scheme $H_g(\P^{g-1})$.
Namely, let $\wt{\MM}_g$ denote the $\PGL_g$-bundle over $\MM_g$ corresponding to a choice of a basis
of $H^0(C,\om_C)$, up to rescaling, let $\wt{\Hyp}_g\sub\wt{\MM}_g$ denote the hyperelliptic locus,
and let $\ov{R}_g$ denote the closure of the ribbon locus in $H_g(\P^{g-1})$.
We prove that away from characteristic $2$, there is an open immersion of the two blow-ups,
$$[\Bl_{\wt{\Hyp}_g}\wt{\MM}_g]^{\coarse}\hra \Bl_{\ov{R}_g}H_g(\P^{g-1})$$
whose image can be explicitly described (see Theorem \ref{two-blow-ups-thm}).
The proof is based on the analysis of equations of hyperelliptic curves and ribbons embedded into the weighted projective
space with $g$ homogeneous coordinates of weight $1$ and $g-2$ coordinates of weight $2$.
In addition, we prove that the blow up $\Bl_{\wt{\Hyp}_g}\wt{\MM}_g$ is precisely the graph of the rational
map from $\wt{\MM}_g$ to the canonical Hilbert scheme (see Corollary \ref{regular-map-cor}).

\medskip

\noindent
{\it Acknowledgments}.  The research of A.P. is partially supported by the
NSF grant DMS-1700642 and by the Russian Academic Excellence Project
`5-100'.  This collaboration began at an administrative meeting at MSRI
when the second author overheard the first author asking David Eisenbud the
question about hyperelliptic limits of quadrics through canonical curves; after several false starts, the
authors eventually implemented David's original suggestion that the answer should involve ribbons.

\section{Hyperelliptic limits of canonical curves in the Hilbert scheme}

\subsection{Ribbons}\label{ribbons-sec}

Recall that a {\it ribbon} is a non-reduced scheme $\RR$ such that one has an exact sequence
$$0\to \LL\to \OO_\RR\to \OO_{\RR_{\red}}\to 0$$ 
where $\LL^2=0$ and $\LL$, viewed as an $\OO_{\RR_{\red}}$-module, is a line bundle on $\RR_{\red}$.
We will only consider ribbons such that $\RR_{\red}=\P^1$. In this case the arithmetic genus of $\RR$ is $g\ge 0$
if and only if $\LL\simeq \OO_{\P^1}(-g-1)$.
The study of ribbons was initiated in \cite{BE} which remains the main reference for a basic background on them.

\begin{example}({\it Hyperelliptic (=split) ribbons})\label{hyperell-ribbon-ex}
For every $g\ge 0$, 
there is a unique, up to isomorphism, ribbon $\RR_0$ (with reduced scheme $\P^1$), admitting a degree $2$ map to $\P^1$
(see \cite[Sec.\ 2]{BE}). Namely, it has 
$$\OO_{\RR_0}=\OO_{\P^1}\oplus \OO_{\P^1}(-g-1),$$
the trivial extension of $\OO_{\P^1}$ by the square-zero ideal $\LL=\OO_{\P^1}(-g-1)$.
\end{example}

\begin{lem}\label{hyper-ribbon-deform-lem}
For any ribbon $\RR$ of genus $g$, with reduced scheme $\P^1$, there exists a flat family $\wt{\RR}$ over $\A^1$,
together with a closed embedding $\P^1\times\A^1\sub \wt{\RR}$,
with the fiber $\RR$ outside $0\in \A^1$ and with the fiber $\RR_0$ over $0$ (equipped the natural embedding of $\P^1$
into them). 
\end{lem}

\Pf . Let 
$$0\to \LL\to \Om_{\RR}|_{\P^1}\to \Om_{\P^1}\to 0$$
be the restricted contangent sequence of $\RR$, and let $e\in \Ext^1_{\P^1}(\Om_{\P^1},\LL)$
be the corresponding extension class. Note that \cite[Thm.\ 1.2]{BE} provides a simple construction recovering $\RR$
from the class $e$. Now let us consider the extension sequence
$$0\to \LL\boxtimes\OO_{\A^1}\to \EE\to \Om_{\P^1}\boxtimes\OO_{\A^1}\to 0$$
over $\P^1\times \A^1$ corresponding to the extension class $t\cdot e$, where $t$ is the coordinate on $\A^1$.
We define the sheaf of abelian groups $\OO_{\wt{\RR}}$ on $\P^1\times\A^1$ to be the fibered product
\begin{diagram}
\OO_{\wt{\RR}}&\rTo{}& \OO_{\P^1\times\A^1}\\
\dTo&&\dTo{d_{\P^1}}\\
\EE&\rTo{}& \Om_{\P^1}\boxtimes\OO_{\A^1}
\end{diagram}
and equip it with the ring structure in a natural way. It is easy to check that $\wt{\RR}$ is a family of
ribbons over $\A^1$ with the desired properties.
\ed

It is shown in \cite{BE} that for every non-hyperelliptic ribbon $\RR$ the canonical embedding identifies $\RR$
with a subscheme of $\P^n$ (where $n=g-1$) supported on a rational normal curve $C_r\sub\P^n$.
We will refer to the obtained subschemes as {\it canonical ribbons} in $\P^n$. 
More explicitly, 
a canonical ribbon $\RR\sub\P^n$ corresponds to the quotient of $\OO_{\P^n}/\II_{C_r}^2$ associated with a surjective map
$$\II_{C_r}/\II_{C_r}^2=N^\vee\to \om_{C_r}(-1),$$
so that we have an exact sequence
\begin{equation}\label{ribbon-str-sh-seq}
0\to \om_{C_r}(-1)\to \OO_\RR\to \OO_{C_r}\to 0.
\end{equation}
Note that such $\RR$ has the same Hilbert polynomial (in fact Hilbert
series) as a canonical curve of genus $n+1$ in $\P^n$.

The following identification of the conormal bundle to the rational normal curve is well known 
(the proof in characteristic zero can be found as \cite[Prop.\ 5A.2]{BE}).

\begin{lem}\label{normal-bundle-lem}
Let $\iota=\iota_n:\P^1\to \P^n$ denote the Veronese embedding, so that $C_r=\iota(\P^1)$.
Then the conormal bundle of $\P^1\simeq C_r$ in $\P^n$
is given by
$$N^\vee\simeq H^0(\P^1,\OO(n-2))\ot \OO_{\P^1}(-n-2),$$
so that $\iota_*N^\vee\simeq \om_{C_r}(-1)^{\oplus n-1}.$
\end{lem}

\Pf . Let $S=\bigoplus_{m\ge 0}H^0(\P^1,\OO(m))=k[x_0,x_1]$.
Pulling back the Euler sequence
$$0\to \Om_{\P^n}\to H^0(\P^n,\OO(1))\ot \OO_{\P^n}(-1)\to \OO_{\P^n}\to 0$$
by $\iota=\iota_n$, we get that the sheaf $\iota^*\Om_{\P^n}$ is isomorphic to the sheaf $\wt{M}$
associated with the graded $S$-module $M$ that fits into an exact sequence
$$0\to M\to S_n\ot S(-n)\rTo{\can} S\to 0$$
where $\can$ maps generators $S_n$ identically to $S_n$.
Furthermore, the sheaf $\Om_{\P^1}$ is the localization of a similarly defined $S$-module $M(\P^1)$ 
and the natural map $\iota^*\Om_{\P^n}\to \Om_{\P^1}$
is induced by the leftmost vertical map in the diagram
\begin{diagram}
0&\rTo{}& M&\rTo{}& S_n\ot S(-n)&\rTo{\can}& S&\rTo{}& 0\\
&&\dTo{}&&\dTo{\de}&&\dTo{n\cdot\id}\\
0&\rTo{}& M(\P^1)&\rTo{}& S_1\ot S(-1)&\rTo{}& S&\rTo{}& 0
\end{diagram}
Here the middle arrow $\de$ is induced by the map
$$S_n\to S_{n-1}\ot S_1: du_i\mapsto i x_0^{i-1}x_1^{n-i}\ot dx_0+(n-i)x_0^ix_1^{n-i-1}\ot dx_1,$$
where $u_i$ are homogeneous coordinates on $\P^n$ such that $\iota^*u_i=x_0^ix_1^{n-i}$.
It follows that the conormal sheaf $N^\vee$ is identified with the localization of the graded $S$-module
$\ker(\de)\cap\ker(\can)$. 

We claim that the $S$-module $\ker(\can)$ is generated in degree $n+1$ by the elements
$$du_i\ot x_0- du_{i+1}\ot x_1, \ \ 0\le i<n.$$
Indeed, an element $\sum du_i\ot c_i\in \ker(\can)_m$ satisfies the equation
$$\sum_{0\le i\le n} c_ix_0^ix_1^{n-i}=0$$
in $S_{n+m}$. This equation implies that $c_n$ is divisible by $x_1$, and
$$\sum_{0\le i\le n-2} c_i x_0^ix_1^{n-1-i}+(c_{n-1}+x_0c_n/x_1)x_0^{n-1}=0,$$
thus, we can apply the induction to deduce our claim.

We have for $i<n$,
$$\delta(du_i\ot x_0- du_{i+1}\ot x_1)=-x_0^i x_1^{n-i}\otimes dx_0
+x_0^{i+1}x_1^{n-1-i}\otimes dx_1.$$
Hence, an element $\sum_{0\le i<n} c'_i(du_i\ot x_0-du_{i+1}\ot x_1)$ is in $\ker(\de)$
if and only if
$$\sum_{0\le i<n}c'_i x_0^ix_1^{n-1-i}=0.$$
It follows that the graded module $\ker(\de)\cap\ker(\can)$ is generated by $n-1$ linearly independent elements
of degree $n+2$,
$$\beta_i=x_0(du_i\ot x_0-du_{i+1}\ot x_1)-x_1(du_{i+1}\ot x_0-du_{i+2}\ot x_1), \ \ 0\le i<n-1,$$
which can be identified with the basis $(x_0^ix_1^{n-2-i})$ of $S_{n-2}$.
\ed

Thus, canonical ribbons that are thickenings of $C_r=\iota_n(\P^1)$ are
parametrized by nonzero linear functionals on $H^0(\P^1,\OO(n-2))$. We
denote by $\RR_\la$ the ribbon associated with $\la\in
H^0(\P^1,\OO(n-2))^*$.

Now we are going to describe quadrics through canonical ribbons.

For a finite-dimensional vector space $V$ we denote by $S_2V$ the subspace
of symmetric tensors in $V^{\ot 2}$ (whereas the usual symmetric square,
$S^2V$, is the quotient of $V^{\ot 2}$), so that we have a universal
quadratic form $V\to S_2V$.

\begin{lem}\label{quadratic-map-lem} 
(i) Let $(x_0, x_1)$ be a basis of $H^0(\P^1,\OO(1))$. Consider the quadratic map
$$Q:H^0(\P^1,\OO(n-2))\to S^2H^0(\P^1,\OO(n)): f\mapsto (x_0^2f)(x_1^2f)-(x_0x_1f)^2.$$
This gives an embedding 
$$S_2H^0(\P^1,\OO(n-2))\to S^2H^0(\P^1,\OO(n))=H^0(\P^n,\OO(2))$$ 
whose image is equal to $H^0(\P^n,\II_{C_r}(2))$.

\noindent
(ii) Let us consider the natural projection
$$\phi_2:H^0(\P^n,\II_{C_r}(2))\to H^0(\P^n,\II_{C_r}/\II_{C_r}^2(2))\simeq 
H^0(\P^1,\OO(n-2))\ot H^0(\P^1,\OO(n-2)),$$
where we use the isomorphism of Lemma \ref{normal-bundle-lem}.
Then the composition of $\phi_2$ with the isomorphism $S_2H^0(\P^1,\OO(n-2))\rTo{\sim}H^0(\P^n,\II_{C_r}(2))$
constructed in (i), is the natural embedding of the subspace of symmetric tensors.
\end{lem}

\Pf . By the dimension count, for (i) it is enough to check that $Q$ is injective. Thus, it is enough to prove (ii).
In the notation of the proof of Lemma \ref{normal-bundle-lem} we have the map
$$H^0(\II_{C_r}(2))\to H^0(N^\vee(2))$$
sends $u_{i+2}\ot u_i-u_{i+1}^2\in H^0(\II_{C_r}(2))$ to 
$$du_{i+2}\ot \iota^*u_i+du_i\ot \iota^*u_{i+2}-2du_{i+1}\ot \iota^*u_{i+1}=x_0^ix_1^{n-2-i}\beta_i.$$
It remains to note that 
$$Q(u_i)=u_{i+2}u_i-u_{i+1}^2,$$
so we deduce the required compatibility.
\ed

\begin{prop}\label{deg-qu-prop}
(i) For $\la\in H^0(\P^1,\OO(n-2))^*$, 
a quadric $q\in S_2H^0(\P^1,\OO(n-2))\simeq H^0(\P^n,\II_{C_r}(2))$ vanishes on a canonical ribbon $\RR_\la$ if and only if
$\iota_\la(q)=0$, where 
$$\iota_\la: S_2H^0(\P^1,\OO(n-2))\to H^0(\P^1,\OO(n-2))$$
is induced by the map $\la\ot\id: H^0(\P^1,\OO(n-2))^{\ot 2}\to H^0(\P^1,\OO(n-2))$.

\noindent
(ii) A quadric $q\in S_2H^0(\P^1,\OO(n-2))$ vanishes on some canonical ribbon if and only if $\det(q)=0$.
\end{prop}

\Pf . (i) By definition, we have a morphism of exact sequences
\begin{diagram}
0&\rTo{}& \II_{C_r}&\rTo{}&\OO_{\P^n}&\rTo{}&\OO_{C_r}&\rTo{}&0\\
&&\dTo{}&&\dTo{}&&\dTo{\id}\\
0&\rTo{}& \om_{C_r}(-1)&\rTo{}& \OO_{\RR_\la}&\rTo{}& \OO_{C_r}&\rTo{}& 0
\end{diagram}
where the left vertical arrow is induced by $\la$.
Now $q$ vanishes on $\RR_\la$ if an only if it lies in the kernel of the map
$$H^0(\P^n,\II_{C_r}(2))\to H^0(\P^n,\II_{C_r}/\II_{C_r}^2(2))\rTo{\la} H^0(C_r,\om_{C_r}(1)).$$
By Lemma \ref{quadratic-map-lem}(ii), the latter map can be identified with
the composition
$$S_2H^0(\P^1,\OO(n-2))\hra H^0(\P^1,\OO(n-2))^{\ot 2}\rTo{\la\ot\id} H^0(\P^1,\OO(n-2)),$$
which is exactly $\iota_\la$.

\noindent
(ii) This immediately follows from (i): $q$ is degenerate if and only if $\iota_\la(q)=0$ for some $\la\neq 0$.
\ed

\subsection{Ribbons as hyperelliptic limits}

Let $p:C\to \Spec(R)$ be a family of smooth curves of genus $g\ge 3$ over a dvr $R$, such that
the general fiber $C_K$ is non-hyperelliptic, while the special fiber $C_k$ is hyperelliptic.

We consider the relative projective space $\P(\VV)$ associated to
the Hodge bundle $\VV=p_*(\om_{C/R})$ over $\Spec(R)$.
Let $\HH_g$ denote the relative Hilbert scheme of subschemes in $\P(\VV)$ having the same Hilbert scheme
as a canonical curve of genus $g$.
Since $C_K$ is non-hyperelliptic, the image of the canonical embedding 
$$f_K:C_K\to \P(\VV_K)$$ 
gives us a point $[f_K(C_K)]\in\HH_g(K)$.
Since $\HH_g$ is proper over $R$, it extends to a unique point in $\HH_g(R)$.
We will denote the corresponding point in $\HH_g(k)$ as $\lim[f(C)]$.

\begin{lem}\label{limit-lem} 
The point $\lim[f(C)]$ corresponds to a canonical ribbon supported on a rational normal curve
$C_r$ which is the image of the canonical map of $C_k$.
\end{lem}

\Pf . Let $f:C\to \P(\VV)$ be the canonical map for our family,
Then $F=f_*\OO_C$ is a coherent sheaf on $\P(\VV)$, flat over $R$, with generic member being
the structure sheaf of the canonical image.
Furthermore, the restriction of $F$ to $\P(\VV_k)$ is 
$$F\ot_R k\simeq \OO_{C_r}\oplus \omega_{C_r}(-1),$$
where $C_r\sub \P(\VV_k)$ is a rational normal curve.

Let $F'\sub F$ be the image of the canonical morphism $\OO_{\P(\VV)}\to F$. Then $F'$ is still flat over $R$,
and is given by the structure sheaf on the canonical image over $K$. Thus, $F'\ot_R k=\OO_{\ov{C}}$
is precisely the structure sheaf of $\ov{C}=\lim[f(C)]$.

Let us consider the cokernel $G$ of the map $\OO_{\P(\VV)}\to F$, so that we have an exact sequence
\begin{equation}\label{F'-ex-seq}
0\to F'\to F \to G\to 0
\end{equation}
Note that $G$ is supported on $\P(\VV_k)$ and 
$$G\ot_R k\simeq \coker(\OO_{\P(\VV_k)}\to F|_{\P(\VV_k)})\simeq \omega_{C_r}(-1).$$ 
It follows that $G$ is supported on $C_r\sub\P(\VV_k)$.

Now we have the long exact sequence obtained from \eqref{F'-ex-seq},
$$\Tor_1^R(G,k)\to F'\ot_R k\to F\ot_R k\to G\ot_R k\to 0.$$
Thus, from our identification of $F\ot_R k$ we get an exact sequence
$$\Tor_1^R(G,k)\to F'\ot_R k\to \OO_{C_r}\to 0$$
Since $G$ is supported on $C_r$, the same is true for $\Tor_1^R(G,k)$.

Hence, we have a surjective morphism of exact sequences 
of coherent sheaves on $\P(\VV_k)$,
\begin{diagram}
0&\rTo{}& \II_{C_r}&\rTo{}&\OO_{\P(\VV_k)}&\rTo{}&\OO_{C_r}&\rTo{}&0\\
&&\dTo{}&&\dTo{}&&\dTo{\id}\\
0&\rTo{}& K&\rTo{}& \OO_{\ov{C}}&\rTo{}& \OO_{C_r}&\rTo{}& 0
\end{diagram}
where $K$ is supported on $C_r$.
Furthermore, by flatness, $K$ 
has the same Hilbert polynomial as $\omega_{C_r}(-1)$.  In particular, 
there is a surjective map from $K$ to a line bundle on $C_r$ of degree at 
most that of $\omega_{C_r}(-1)$, which may thus be identified 
noncanonically with a subsheaf of $\omega_{C_r}(-1)$.  
But any map $\II_{C_r}\to\omega_{C_r}(-1)$ factors
through $\II_{C_r}/\II_{C_r}^2\simeq \om_{C_r}(-1)^{g-2}$ (see Lemma \ref{normal-bundle-lem})
and thus is is surjective. 
Hence, the composition 
$$\II_{C_r}\to K\to \omega_{C_r}(-1)$$ 
is surjective, so that $K\to \omega_{C_r}(-1)$ is 
surjective.  Since $K$ and $\omega_{C_r}(-1)$ have the same Hilbert polynomials, it follows that 
$K\cong \omega_{C_r}(-1)$, so $\ov{C}$ is a canonical ribbon.
\ed


\begin{rmk} 
In the case $g=3$, every conic $C_r\sub\P^2$
is contained in a unique canonical ribbon in $\P^2$, defined by the ideal $\II_{C_r}^2$.
By Lemma \ref{limit-lem}, for every $1$-parameter family which is generically non-hyperelliptic
and has a hyperelliptic curve $C_0$ as a central fiber, the corresponding family in the relative Hilbert scheme has this canonical
ribbon as a limit.
\end{rmk}

\section{Hyperelliptic limits of canonical relations}

\subsection{Obstruction map associated with a hyperelliptic curve}\label{1st-order-obstruction-sec}

\begin{lem}\label{local-obs-lem}
Let $f:A\to C$ be a ring homomorphism with kernel $I$, and let $B=A/I$ with 
induced morphisms $g:A\to B$, $h:B\to C$.  Then given a first-order 
deformation $\wt{C}$ of $C$ and a homomorphism $\wt{f}:A[\epsilon]/\epsilon^2\to \wt{C}$ deforming $f$, 
an element $x\in I$ extends to an element in the kernel of $\wt{f}$ if and only if
$$\lan [\wt{f}], x\ran=0 \text{ in } C/B,$$ 
where $[\wt{f}]$ is the class of $\wt{f}$ in $\Ext^1_C(L_{C/A},C)$,
and the pairing $\lan\cdot,\cdot\ran$ is induced by the composition
$$\Ext^1_C(L_{C/A},C)\to \Hom_C(I/I^2\otimes_B C,C)
      \cong \Hom_B(I/I^2,C)\to \Hom_B(I/I^2,C/B)\cong \Hom_A(I,C/B),$$
where $C/B$ is the quotient as a $B$-module.
Furthermore, the above pairing is determined by 
$$\lan [\wt{f}],x\ran=-\frac{\wt{f}(x)}{\eps}\mod B,$$
where $\wt{f}(x)\in \eps \wt{C}$ and we use the isomorphism $C\rTo{\eps} \eps\wt{C}$.
\end{lem}

\Pf .  Let $J$ be the kernel of the morphism $A[C]\to C$.  There is a short 
exact sequence
$$0\to \Hom_C(\Omega_{A[C]/A}\otimes C,C)
 \to \Hom_C(J/J^2,C)
 \to \Ext^1_C(L_{C/A},C)
 \to 0$$
and the deformation of $f$ corresponding to $\phi\in \Hom_C(J/J^2,C)$ is 
given by taking $\wt{C}$ to be the quotient $A[\epsilon][C]/(J',\epsilon^2)$, 
where $J'$ is generated by the elements $j+\phi(j)\epsilon$ and $j\epsilon$ 
for $j\in J$  (see, e.g., Lemma 82.2.3 of the Stacks Project).  An element $i\in I$ deforms iff there is an 
element $i_1\in A$ such that $i+i_1\epsilon$ is in $J'$, or equivalently iff 
$\phi(i)\in B$.  The restriction map
$$\Hom_C(J/J^2,C)\to \Hom_B(I/I^2,C)$$
is the composition of natural maps
$$\Hom_C(J/J^2,C)\to \Ext^1_C(L_{C/A},C)\to \Hom_B(I/I^2,C)$$
and thus the obstruction map is as described.

The formula for the pairing follows from the fact that for $i\in I\sub J$ we have
$$i\equiv -\phi(i)\eps \mod \wt{J}.$$
\ed

\begin{prop}\label{general-obstr-prop} 
Let $f:X\to S$ be an affine morphism of schemes with scheme-theoretic image
$i:Y\to S$ and residual factor $g:X\to Y$, and let $\wt{f}:\wt{X}\to
S[\eps]/(\eps^2)$ be a first order deformation of $f$ of class
$[\wt{f}]\in \Ext^1_X(L_{X/S},\sO_X)$.

\noindent
(i) For any affine open subset 
$U\subset S$, the obstruction to extending an element $x\in\Gamma(U;I_Y)$ to 
an element in the kernel of 
$$\wt{f}^*:\Gamma(U,\sO_S[\eps]/(\eps^2))\to \Gamma(U,\wt{f}_*\sO_{\wt{X}})$$
is given by the pairing
$$\lan [\wt{f}], x\ran:=\obs([\wt{f}])(x)\in \Gamma(U;f_*\sO_X/i_*\sO_Y),$$
coming from the composition
\begin{equation}\label{general-obstruction-map-eq}
\obs:\Ext^1_X(L_{X/S},\sO_X)
\to
\Hom_X(g^*(I_Y/I_Y^2),\sO_X)
\to
\Hom_S(I_Y,f_*\sO_X/i_*\sO_Y).
\end{equation}

\noindent
(ii) Let $\LL$ be a line bundle on $S$.  Assume that $H^1(S,I_Y\ot\LL)=0$. Then the obstruction to extending 
$x\in H^0(S,I_Y\ot\LL)$
to an element in the kernel of
$$\wt{f}^*:H^0(S,\LL[\eps]/(\eps^2))\to \Gamma(\wt{X},\wt{f}^*\LL)$$
is given by the pairing 
\begin{equation}\label{line-bun-obstr-pairing}
\Ext^1_X(L_{X/S},\sO_X)\times H^0(S,I_Y\ot\LL)\to H^0(S,f_*\sO_X/i_*\sO_Y\ot \LL):([\wt{f}],x)\mapsto
\lan [\wt{f}],x\ran:=\obs([\wt{f}])(x).
\end{equation}
We also have
\begin{equation}\label{line-bun-obstr-pairing-formula}
\lan \wt{f},x\ran=-\frac{\wt{f}^*(x)}{\eps} \mod \sO_Y.
\end{equation}
\end{prop}

\Pf . (i) This follows immediately from Lemma \ref{local-obs-lem}.

\noindent
(ii) If an extension of $x$ exists then by (i), $\obs(\xi)|_U=0$ for every open affine $U\sub S$ over which $\LL$ is trivial, Hence,
$\obs(\xi)=0$. 

Conversely, assume that $\obs(\xi)=0$. Pick an open affine covering
$S=\cup_i U_i$ such that $\LL|_{U_i}$ is trivial.  By part (i), over each
$U_i$ there exists an extension $x_i\in \Gamma(U_i,\LL[\eps]/(\eps^2))$ of
$x$, such that $\wt{f}^*(x_i)=0$. Over the intersections $U_i\cap U_j$ we
have $x_i-x_j=\eps y_{ij}$, for some $y_{ij}\in \Gamma(U_i\cap U_j,\LL)$
such that $f^*(y_{ij})=0$. Thus, we have $y_{ij}\in \Gamma(U_i\cap
U_j,I_Y\ot\LL)$, so $(y_{ij})$ is a $1$-cocycle with values in
$I_Y\ot\LL$. By assumption, it is a coboundary, so we can correct our
extensions $x_i$ over $U_i$ so that they glue into a global section of
$\LL[\eps]/(\eps^2)$.

The last formula for the pairing $\lan [\wt{f}],x\ran$ comes from a similar formula in Lemma \ref{local-obs-lem}.
\ed

We can apply the above Proposition to the canonical map $f:C\to \P^{g-1}$ where $C$ is
a hyperelliptic curve of genus $g$, which factors as the composition
$$C\rTo{g} \P^1\rTo{i}\P^{g-1},$$
with $g$ a double covering, and $i$ is the embedding as a rational normal curve.
Let $I_{\P^1}$ denote the ideal of $i(\P^1)$ in $\sO_{\P^{g-1}}$. 
Then we see that the obstruction map \eqref{general-obstruction-map-eq} has form
$$\obs:\Ext^1_C(L_{C/\P^{g-1}},\sO_C)
\to
\Hom_{\P^{g-1}}(I_{\P^1},f_*\sO_C/i_*\sO_{\P^1})
\cong
\Hom_{\P^{g-1}}(I_{\P^1},i_*\det(g_*\sO_C)).$$

We can recompute this map in terms of deformations of $C$.  Since 
the natural map
$$\Hom_{\P^{g-1}}(I_{\P^1}/I_{\P^1}^2,i_*\det(g_*\sO_C))
\to
\Ext^1_{\P^1}(\Omega_{\P^1},\det(g_*\sO_C))$$
is an isomorphism, it suffices to consider the induced map
$$\Ext^1_C(L_{C/\P^{g-1}},\sO_C)
\to
\Ext^1_{\P^1}(\Omega_{\P^1},\det(g_*\sO_C)),$$
which in turn arises as the composition
$$\Ext^1_C(L_{C/\P^{g-1}},\sO_C)
\to
\Hom_C(g^*(I_{\P^1}/I_{\P^1}^2),\sO_C)
\to
\Ext^1_C(g^*\Omega_{\P^1},\sO_C)
\to
\Ext^1_{\P^1}(\Omega_{\P^1},g_*\sO_C/\sO_{\P^1})$$
of natural maps.

Apart from the last map, this comes from the composition
$$g^*L_{\P^1/k}\to g^*L_{\P^1/\P^{g-1}}\to L_{C/\P^{g-1}}$$
of natural maps of cotangent complexes, which can also be factored as
$$g^*L_{\P^1/k}\to L_{C/k}\to L_{C/\P^{g-1}},$$
allowing us to describe the map as
$$\Ext^1_C(L_{C/\P^{g-1}},\sO_C)
\to
\Ext^1_C(L_{C/k},\sO_C)
\to
\Ext^1_C(g^*\Omega_{\P^1},\sO_C)
\to
\Ext^1_{\P^1}(\Omega_{\P^1},g_*\sO_C/\sO_{\P^1})$$

It follows immediately that we can rewrite the obstruction map in terms 
of deformations of $C$ as the composition:
\begin{equation}\label{hyperell-obstruction-map-composition}
\Ext^1_C(L_{C/k},\sO_C)
\to
\Ext^1_C(g^*\Omega_{\P^1},\sO_C)
\to
\Ext^1_{\P^1}(\Omega_{\P^1},g_*\sO_C/\sO_{\P^1})
\cong
\Hom_{\P^{g-1}}(I_{\P^1},f_*\sO_C/i_*\sO_{\P^1}).
\end{equation}

\begin{prop}\label{hyperell-obs-prop} 
Let $C$ be a smooth hyperelliptic curve of genus $g$. The composition \eqref{hyperell-obstruction-map-composition}
gives a surjective obstruction map
$$\obs:H^1(C,T_C)\simeq\Ext^1_C(L_{C/k},\OO_C)\to \Hom_{\P^{g-1}}(I_{\P^1}, i_*\det(g_*\OO_C))$$
such that for $n\ge 0$, an element $x\in H^0(\P^{g-1},I_{\P^1}(n))$ deforms along the deformation 
of the canonical map $C\to \P^{g-1}$ corresponding to a first order deformation 
$v\in \Ext^1_C(L_{C/k},\OO_C)$ if and only if $\obs(v)(x)=0$.
\end{prop}

\Pf . The fact that the obstruction is given in this way follows from Proposition \ref{general-obstr-prop} and
from the surjectivity of the map $\Ext^1_C(L_{C/\P^{g-1}},\sO_C)\to\Ext^1(C,L_{C/k},\OO_C)$.

The surjectivity of the obstruction map follows from the first description of the obstruction: the morphism
$$\Hom_{\P^{g-1}}(I_{\P^1},f_*\sO_C)
\to
\Hom_{\P^{g-1}}(I_{\P^1},f_*\sO_C/i_*\sO_{\P^1})$$
is surjective since
$\Ext^1_{\P^{g-1}}(I_{\P^1},i_*\sO_{\P^1})=0$,
and the morphism
$$\Ext^1_C(L_{C/\P^{g-1}},\sO_C)
\to
\Hom_C(g^*(I_{\P^1}/I_{\P^1}^2),\sO_C)$$
is surjective since $L_{C/\P^1}$ is a sheaf, so
$\Ext^2_C(L_{C/\P^1},\sO_C)=0$.
\ed

From Lemma \ref{normal-bundle-lem} we get an isomorphism 
$$\Hom_{\P^{g-1}}(I_{\P^1}, i_*\det(g_*\OO_C))\simeq\Hom_{\P^1}(N^\vee,\OO(-g-1))\simeq H^0(\P^1,\OO(g-3))^\vee,$$
so the obstruction space has dimension $g-2$, which is the same as the codimension of the hyperelliptic locus.
Note that $\obs$ vanishes on the tangent space to the hyperelliptic locus, so by the above dimension
count, we get the following Corollary.

\begin{cor}\label{ker-obs-cor} 
The kernel of $\obs$ is equal to the tangent space to the hyperelliptic locus, and the map $\obs$ induces an isomorphism
\begin{equation}\label{obs-normal-space-eq}
\obs:N_C\rTo{\sim} H^0(\P^1,\OO(g-3))^\vee,
\end{equation}
where $N_C$ is the normal space to the hyperelliptic locus in $\MM_g$ at $C$. 
\end{cor} 

Let us now consider the special case of elements in $H^0(I_{\P^1}(2))$.
Recall that we have an identification 
\begin{equation}\label{quadrics-identification-eq}
S_2H^0(\P^1,\OO(g-3))\rTo{\sim} H^0(I_{\P^1}(2)):q\mapsto x_q
\end{equation}
(see Lemma \ref{quadratic-map-lem}(i)).

\begin{cor}\label{1st-order-obstruction-cor}
Let $C$ be a smooth hyperelliptic curve of genus $g$, and let $x_q\in H^0(I_{\P^1}(2))$ be a quadratic
relation corresponding to $q\in S_2H^0(\P^1,\OO(g-3))$. Then the relation $x_q$ deforms to the first order
in a normal direction $v\in N_C$ if and only if $\obs(v)$ is in the kernel of the bilinear form on $H^0(\P^1,\OO(g-3))^\vee$
determined by $q$.
\end{cor}

\Pf . It is enough to check that the natural pairing of $x_q\in
H^0(I_{\P^1}(2))$ with $y\in
\Hom_{\P^{g-1}}(I_{\P^1},i_*\det(g_*\sO_C))\simeq H^0(\P^1,\OO(g-3))^\vee$
is given by contracting the bilinear form given by $q$ with $y$. But this
follows from Lemma \ref{quadratic-map-lem}(ii).  \ed


\subsection{Hyperelliptic limits of canonical quadrics and ribbons}

For every curve $C$ of genus $g$ we consider the morphism
$$\mu_C:S^2H^0(C,\om_C)\to H^0(C,\om_C^{\ot 2}).$$

\begin{definition}
Let $C_0$ be a hyperelliptic curve of genus $g$ over $k$. We say that a quadric
$x_0\in \ker(\mu_{C_0})$ is a {\it limit of canonical quadrics} if there exists a dvr $R$ with the residue field $k$
and the fraction field $K$, and a family of curves $C$ over $R$, and a quadric $x\in\ker(\mu_C)$
such that $(C_k,x_k)\simeq (C_0,x_0)$ and $C_K$ is non-hyperelliptic.
\end{definition}

The canonical map $C_0\to \P^{g-1}$ of a hyperelliptic curve factors through
the double covering $C_0\to \P^1$. Thus, we have 
an identification
$$\ker(\mu_{C_0})\simeq H^0(\P^{g-1},I_{\P^1}(2)),$$
and we can use the map $q\mapsto x_q$ (see \eqref{quadrics-identification-eq})
that identifies the latter space with $S_2H^0(\P^1,\OO(g-3))$.

\begin{thm}\label{quadric-hyperell-thm} 
Let $g\ge 3$. A quadric $x_q\in \ker(\mu_{C_0})$ is a limit of canonical quadrics if and only if the corresponding
element $q\in S_2H^0(\P^1,\OO(g-3))$ is degenerate.
\end{thm}

\Pf .  The proof is based on the use of the first order obstruction defined
in Section \ref{1st-order-obstruction-sec}.  Namely, recall that we have an
identification $\obs$ of the normal space $N_{C_0}$ to the hyperelliptic
locus at $C_0$ with $H^0(\P^1,\OO(g-3))^\vee$ (see
\eqref{obs-normal-space-eq}), such that $x_q$ has a first order deformation
to a canonical quadric in the direction of $v\in T_{[C_0]}\MM_g$ if and
only if $\obs(v)$ is in the kernel of $q$ (see Corollary
\ref{1st-order-obstruction-cor}).

Let $X$ denote the stack of pairs $(C,x)$ where $C$ is a curve and $x\in\ker(\mu_C)$.
We can view $X$ as a closed substack in the total space $\tot(W^0)$
of the bundle $W^0=S^2\pi_*\om_C$ over $\MM_g$, given as the zero locus of the section $s=\mu_C$ of $p^*W^1$,
where $W^1=\pi_*(\om_C^{\ot 2})$ and
$p:\tot(W^0)\to \MM_g$ is the projection.

We claim that in fact $s$ is a regular section, and so $X$ is a local complete intersection.

Indeed, this follows from the fact that $X$ has two irreducible components
of the same dimension (equal to the expected dimension of $s=0$)
$$N=3g-3+\dim\ker(\mu_C)=2g-1+\dim\ker(\mu_{C_0})$$ 
where $C$ is non-hyperelliptic and $C_0$ is hyperelliptic. The first component is the closure
of the locus where $C$ is non-hyperelliptic, while the second component has $C_0$ hyperelliptic and $x\in\ker(\mu_{C_0})$
arbitrary. The equality of dimensions follows from the fact that $\mu_C$ is surjective for non-hyperelliptic $C$, while
$\dim\coker(\mu_{C_0})=g-2$ for hyperelliptic curve $C_0$.

Next, we claim that each component of $X$ is generically reduced. Indeed, this is clear on the non-hyperelliptic component.
Now let $C_0$ be a hyperelliptic curve, and let $x=x_q$ be the canonical quadric corresponding to a nondegenerate
$q\in S_2H^0(\P^1,\OO(g-3))$.
We claim that $X$ is smooth at $(C_0,x_q)$.

Indeed, the tangent space to $(C_0,x_q)$ consists of pairs $(v,\wt{x})$,
where $v\in T_{[C_0]}\MM_g$ and $\wt{x}$ is a first order deformation of
$x_q$ along $v$.
Note that for given $v$ the set of liftings is a torsor for $\ker(\mu_{C_0})$. 
Thus, the dimension of the tangent space to $X$ at $(C_0,x)$ is equal to $\dim\ker(\mu_{C_0})+d$, where
$d$ is the dimension of the space of $v$ such that $x_q$ deforms
along $v$. Since $q$ is nondegenerate, $x_q$ deforms along $v$ if and only if $\obs(v)=0$, i.e.,
$v$ is tangent to the hyperelliptic locus. Hence, $d=2g-1$, and
we deduce that $X$ is smooth at $(C_0,x_q)$.

Since $X$ is l.c.i, so has no embedded components, we obtain that $X$ is
reduced.  Furthermore, the hyperelliptic component of $X$ is smooth, since
it is a vector bundle over the hyperelliptic locus.  Hence, $x_q$ is a
limit of canonical quadrics if and only if $X$ is singular at $(C_0,x_q)$.

As we have seen above, if $q$ is nondegenerate then $X$ is smooth at $(C_0,x_q)$, so $x_q$ is not a limit
of canonical quadrics. On the other hand, if $q$ is degenerate, then there exists a normal direction to the hyperelliptic
locus $v$, such that $\obs(v)$ is in the kernel of $q$. Hence, the dimension of the space of tangent vectors $v$ in
$T_{[C_0]}\MM_g$ 
such that $x_q$ deforms along $v$, is $>2g-1$. This implies that
the dimension of the tangent space $T_{(C_0,x_q)}X$ is $>N$,
so $X$ is singular at $(C_0,x_q)$. Hence, $x_q$ is a limit of canonical quadrics.
\ed

\subsection{Hyperelliptic limits of higher degree relations}

Now we want to consider hyperelliptic limits of relations in
$$\ker(\mu^d_C:S^dH^0(C,\omega_{C})\to H^0(C,\omega_{C}^{\ot d})).$$
As in the quadratic case, we say that for a hyperelliptic curve $C_0$ an element $f_0\in \ker(\mu^d_{C_0})$
is a {\it limit of canonical relations} if it is obtained by a specialization from an element of $\ker(\mu^d_C)$
for a family $C/R$ over a dvr $R$ such that $C_k\simeq C_0$ and $C_K$ is non-hyperelliptic.

Let $C_0$ be a hyperelliptic curve of genus $g$. Then we have an identification
$$\ker(\mu^d_{C_0})\simeq H^0(\P^{g-1},\II_{C_r}(d)),$$
where $C_r\sub\P^{g-1}$ is the rational normal curve obtained as the image of the canonical morphism of $C_0$.

Let us consider the natural projection
\begin{equation}\label{phi-d-eq}
\phi_d:H^0(\P^{g-1},\II_{C_r}(d))\to H^0(\P^{g-1},\II_{C_r}/\II_{C_r}^2(d))\simeq 
H^0(\P^1,\OO(g-3))\ot H^0(\P^1,\OO((d-1)(g-1)-2)).
\end{equation}
We can view the elements of the target space as $(g-2)\times
((d-1)(g-1)-1)$-matrices, and in particular talk about their rank.


\begin{lem}\label{ribbon-sur-lem} Let $C_r\sub \P^n$ be a rational normal curve. The map 
$$\phi_d:H^0(\P^n,\II_{C_r}(d))\to H^0(\P^n,\II_{C_r}/\II_{C_r}^2(d))$$
is surjective for $d\ge 3$.
\end{lem}

\Pf . Is enough to prove the surjectivity of the composed map
\begin{align*}
&H^0(\P^n,\OO(d-2))\ot H^0(\P^n,\II_{C_r}(2))\to 
H^0(\P^n,\OO(d-2))\ot H^0(\P^n,\II_{C_r}/\II_{C_r}^2(2))\to\\
&H^0(\P^n,\II_{C_r}/\II_{C_r}^2(d)),
\end{align*}

By Lemma \ref{quadratic-map-lem}, this reduces to the surjectivity of the composition
\begin{align*}
&H^0(\P^1,\OO_{\P^1}(n(d-2)))\ot S_2H^0(\P^1,\OO(n-2))\to 
H^0(\P^1,\OO_{\P^1}(n(d-2)))\ot H^0(\P^1,\OO(n-2))^{\ot 2}\to\\ 
&H^0(\P^1,\OO(nd-n-2))\ot H^0(\P^1,\OO(n-2)).
\end{align*}

We claim that in fact the composed map
\begin{align*}
&\a_e:H^0(\P^1,\OO(e))\otimes S_2H^0(\P^1,\OO(m))\to H^0(\P^1,\OO(e))\otimes H^0(\P^1,\OO(m))^{\ot 2}\to \\
&H^0(\P^1,\OO(m+e))\otimes H^0(\P^1,\OO(m))
\end{align*}
is surjective for all $e\ge 1$, $m\ge 0$.
We can view it as a degree $e$ component of the map of graded $S$-modules, where $S=\bigoplus_e H^0(\P^1,\OO(e))$, 
$$S\ot S_2(S_m)\to S(m)\ot S_m.$$ Since the module $S(m)_{\ge 1}$ is
generated in degree $1$, it is enough to prove the surjectivity of the
degree $1$ component, $\a_1$.  If $1,x$ are a basis of
$S_1=H^0(\P^1,\OO(1))$, then
$$\a_1(x\ot f^{\ot 2})=xf\ot f, \ \ \a_1(1\ot f^{\ot 2})=f\ot f.$$ 
Now we can use the induction on $m\ge 0$. By the induction assumption, we can assume that all $x^i\ot x^j$ with $i\le m$,
$j\le m-1$ are in the image of $\a_1$. Since 
$$x^i\ot x^m+x^m\ot x^i=\a_1(1\ot (x^i\ot x^m+x^m\ot x^i)),$$ 
this implies that $x^i\ot x^m$ is still in the image of $\a_1$ for $i\le m-1$. Also, $x^m\ot x^m=\a_1(1\ot (x^m)^{\ot 2})$.
Since 
$$x^{m+1}\ot x^j+x^{j+1}\ot x^m=\a_1(x\ot (x^m\ot x^j+x^j\ot x^m)),$$
we deduce that $x^{m+1}\ot x^j$ is in the image of $\a_1$ for $j\le m$.
Finally, $x^{m+1}\ot x^m=\a_1(1\ot (x^m)^{\ot 2})$.
\ed

\begin{lem}\label{ribbon-coh-lem} 
Let $\RR$ be a canonical ribbon in $\P^n$. Then for $d\ge 2$, $H^1(\RR,\OO_\RR(d))=0$ and 
$h^0(\RR,\OO_\RR(d))=(2d-1)n$. Also, the natural map
$$H^0(\P^n,\OO(d))\to H^0(\RR,\OO(d))$$
is surjective.
\end{lem}

\Pf . The vanishing of $H^1$ follows from the exact sequence \eqref{ribbon-str-sh-seq}, twisted by $\OO(d)$ for $d\ge 2$.
In the case $d=2$ the required surjectivity follows from the dimension count using our description of quadrics through $\RR$.
In the case $d\ge 3$, Lemma \ref{ribbon-sur-lem} implies the surjectivity of the left vertical arrow in the morphism of
exact sequences
\begin{diagram}
0&\rTo{}& H^0(\II_{C_r}(d))&\rTo{}&H^0(\OO_{\P^n}(d))&\rTo{}&H^0(\OO_{C_r}(d))&\rTo{}&0\\
&&\dTo{}&&\dTo{}&&\dTo{\id}\\
0&\rTo{}& H^0(\om_{C_r}(d-1))&\rTo{}& H^0(\OO_{\RR_\la}(d))&\rTo{}& H^0(\OO_{C_r}(d))&\rTo{}& 0
\end{diagram}
Hence, the middle vertical arrow is also surjective.
\ed

\begin{thm}\label{rel-hyperell-thm}
Let $g\ge 3$, $d\ge 2$. For a hyperelliptic curve $C_0$, a relation $f_0\in\ker(\mu^d_{C_0})$
is a limit of canonical relations if and only if
$\phi_d(f_0)$ has rank $<g-2$.
\end{thm}

\Pf . First, we observe that $\phi_d(f_0)$ has rank $<g-2$ if and only if there exists some $\la\in H^0(\P^1,\OO(g-3))^*$
such that $(\la\ot\id)(\phi_d)=0$. The latter condition is equivalent to $f_0$ being in the kernel of the map
$$H^0(\P^{g-1},\II_{C_r}(d))\to H^0(\RR_{\la},\OO),$$
where $\RR_\la$ is the canonical ribbon corresponding to $\la$ (see the proof of Proposition \ref{deg-qu-prop}(i)).

Now, as in Theorem \ref{quadric-hyperell-thm}, we deduce using Lemma \ref{limit-lem}
that any $f_0$ which is a limit of canonical relations, vanishes on some canonical ribbon $\RR_\la$, and hence
$\phi_d(f_0)$ has rank $<g-2$.

Conversely, suppose $f_0$ vanishes on some canonical ribbon $\RR_\la$.
Let us pick a quadric $q_0\in\ker(\mu^2_{C_0})$, vanishing on $\RR_\la$, for which the corresponding element of
$S_2H^0(\P^1,\OO(g-3))$ has rank $g-3$. 
Then by Theorem \ref{quadric-hyperell-thm}, there exists a family $(C,q)$ over a dvr $R$ deforming $(C_0,q_0)$
with $C_K$ non-hyperelliptic. By Lemma \ref{limit-lem}, the limiting point $\lim[f(C)]$ in the Hilbert
scheme is a canonical ribbon $\RR=\RR_{\la'}$, such that $q_0$ vanishes on $\RR$. But this implies that $\la'$ is proportional
to $\la$, so $\RR=\RR_\la$. 

Now we claim that the $R$-module
$M:=\ker(\mu^d_C)$ is flat and $M\ot_R k\simeq\ker(\mu^d_{C_0})$. 
Indeed, this follows from the vanishing of $H^1(\RR,\OO_\RR(d))$ and from the surjectivity
of the restriction map $H^0(\P^{g-1},\OO(d))\to H^0(\RR,\OO_\RR(d))$ (see Lemma \ref{ribbon-coh-lem}). 

Our element $f_0$ belongs to $M\ot_R k$, hence, it can be lifted to an element $f\in M$. Thus, $f_0$ is a limit
of canonical relations.
\ed

\section{Blow-ups of the hyperelliptic and ribbon loci}

\subsection{Rational maps to the Grassmannians and Fitting ideals}

Let $X$ be a scheme, $\VV$, $\WW$ vector bundles, and $f:\VV\to\WW$ a morphism,
which is surjective over a dense open subset $U\sub X$. Then it defines a section
$$\si:U\to G_k(\VV),$$
of the relative Grassmanian over $X$ associated with $\VV$, where $k=\rk\VV-\rk\WW$.

Let us consider the $0$th Fitting ideal of $\coker(f)$,
$\Fitt_0(\coker(f))$. By definition, it is the vanishing ideal of the map 
$${\bigwedge}^r(f):{\bigwedge}^r(\VV)\to {\bigwedge}^r(\WW),$$
where $r$ is the rank of $\WW$.

\begin{prop}\label{Fitting-prop} 
Let $\wt{X}\to X$ be the blow-up of
$X$ at $\Fitt_0(\coker(f))$. Then the section $\si$ extends to a closed embedding 
$$\wt{\si}:\wt{X}\to G_k(\VV).$$
Furthermore, if $X$ is integral then $\wt{X}$ is identified with the closure in $G_k(\VV)$
of the rational section provided by $\si$.
\end{prop}

\begin{proof} First, consider the situation when we have a cosection of a vector bundle $F:\VV\to\OO_S$.
Then the surjection of graded algebras
$$S^\bullet(\VV)\to \bigoplus_{n\ge 0}F(\VV)^n$$
gives us an identification of the
blow-up of $X$ in the ideal $F(\VV)\sub\OO_S$ with a closed subset of $\P(\VV^\vee)$.

In our situation, we can apply the above construction to the cosection
$$F:{\bigwedge}^r(\VV)\ot{\bigwedge}^r(\WW)^{-1}\to\OO_S$$
induced by $\we^r(f)$, so we get a closed embedding
$$\wt{\si}:\wt{X}\to \P(\VV^\vee).$$ We claim that in fact it factors
through the Grassmannian $G_k(\VV)$ embedded into $\P(\VV^\vee)$ via the
Pl\"ucker embedding. Indeed, locally the components of $F$ are given by the
$r\times r$ minors of the matrix of $f$, so they satisfy the Pl\"ucker
relations.

The last assertion is clear since $\wt{X}$ is integral and $\Fitt_0(\coker(f))$ is supported on the complement of $U$.
\end{proof}

We will need the following technical assertion.

\begin{lem}\label{Fitt-power-ideal-lem} 
Let $X$ be a scheme, $Y\sub X$ a closed subscheme with the ideal $\II_Y$,
and let $f:\VV\to\WW$ be a morphism of vector bundles over $X$, with $\rk\VV\ge \rk\WW$,
such that $\coker(f|_Y:\VV|_Y\to \WW|_Y)$
is locally free of rank $r$ on $Y$. Then locally near every point of $Y$ there exists a morphism of vector bundles
$f':\VV'\to\WW'$, such that $[\VV\to\WW]$ is quasi-isomorphic to $[\VV'\to\WW']$
and $f'|_Y=0$. In particular, we have
$$\Fitt_0(\coker(f))\sub \II_Y^r.$$
\end{lem}

\Pf . We can assume $X$ to be affine. The fact that $\coker(f_Y)$ is locally free implies that
we have a direct sum decomposition 
$$\WW|_Y=\im(f|_Y)\oplus C,$$
where $C$ is a bundle of rank $r$ on $Y$. 
Let us replace $X$ by a local ring $A$, so that $\II_Y$ corresponds to a proper ideal $I\sub A$.
The $A/I$-modules $\im(f|_Y)$ and $C$ are free, so with respect to suitable bases of $\VV|_Y$ and $\WW|_Y$,
we will have
$$f|_Y=\left(\begin{matrix} 1_p & 0 \\ 0 & 0 \end{matrix}\right),$$
where $p=\rk \im(f|_Y)$.
Thus, for a suitable choice of bases in $\VV$ and $\WW$, we have
$$f=\left(\begin{matrix} 1_p & 0 \\ 0 & M\end{matrix}\right),$$
for some $s\times r$-matrix $M$ with entries in $I$ (where $s=\rk \ker(f|_Y)$). 
The assertion immediately follows from this.
\ed

\subsection{Obstruction associated with a $2$-term complex}\label{2term-obs-sec}

Suppose we are given a two-term complex $W^0\to W^1$ of vector 
bundles on a smooth scheme $X$, and let $Z$ be a smooth subsheme of 
$X$ such that $\und{H}^1(W^\bullet|_Z)$ is a vector bundle, so that 
$\dim(H^0(W^\bullet\otimes \sO_z))$ is constant for $z\in Z$.  

Suppose that $A$
is a local Artin algebra and $I\subset A$ is a square zero ideal, and let 
$z\in Z(A/I)$. Then we know that $H^0(W^\bullet\otimes A/I)$ is a free $A/I$-module.  Given a 
point $x\in X(A)$ extending $z\in X(A/I)$, we want to know which elements of this free 
module extend to $H^0(W^\bullet\otimes A)$, i.e., belong to the image of the map
$$H^0(W^\bullet\otimes A)\to H^0(W^\bullet\otimes A/I).$$
Using the exact sequence of cohomology associated with the exact triple of complexes
$$0\to W^\bullet\ot I\to W^\bullet\ot A\to W^\bullet \ot A/I\to 0$$
we immediately see that the obstruction is given by the coboundary map
$$H^0(W^\bullet\otimes A/I)\to H^1(W^\bullet\otimes I).$$
Note that this map is induced by the element $e_x\in\Ext^1_X(A/I,I)$ given by 
the class of the extension
$$0\to I\to A\to A/I\to 0$$ 
of  $\sO_X$-modules.  The set of 
extensions of $z$ to a point of $X(A)$ is a torsor over $T_X\otimes I$, and the map $x\mapsto e_x$
from this torsor to $\Ext^1_X(A/I,I)$ is a torsor map relative to a homomorphism 
$T_X\otimes I\to \Ext^1_X(A/I,I)$.  The induced map from this torsor to
$$\Hom_{A/I}(H^0(W^\bullet\otimes A/I),H^1(W^\bullet\otimes I))$$
vanishes on the classes corresponding to $x\in Z(A)$, and thus factors 
through a homomorphism from the trivial quotient torsor 
$N_{X/Z}\otimes I$.  

In particular, we can apply this construction to $A=k[\eps]/(\eps^2)$ to get a pairing
$$\kappa:N_{X/Z}|_z\otimes H^0(W^\bullet|_z)\to H^1(W^\bullet|_z)$$
such that $w\in H^0(W^\bullet)$ deforms to the first order in the direction of $v\in N_{X/Z}|_z$ if and only if
$\kappa(v,w)=0$.

\subsection{Blow up of the hyperelliptic locus}

For $d\ge 2$, let us consider the morphism of vector bundles 
\begin{equation}\label{mu-d-moduli-complex}
\mu^d:S^dp_*(\om_{\CC_g/\MM_g})\to p_*(\om_{\CC_g/\MM_g}^{\ot d})
\end{equation}
over $\MM_g$, where $p:\CC_g\to \MM_g$ is the universal curve,
and let $\Hyp_g\sub\MM_g$ denote the hyperelliptic locus.
As is well known, for $d\ge 2$, the map $\mu^d$ is surjective over $\MM_g\setminus\Hyp_g$
and defines a section of the relative Grassmannian associated with $S^dp_*(\om_{\CC_g/\MM_g})$,
which factors through the relative Hilbert scheme $\HH_g/\MM_g$ containing canonical curves of genus $g$.
On the other hand, $\coker(\mu^d|_{\Hyp_g})$ is a vector bundle over $\Hyp_g$ so we can apply the
construction of Section \ref{2term-obs-sec} to the two-term complex \eqref{mu-d-moduli-complex}
and the closed locus $\Hyp_g$.

\begin{lem}\label{two-obstr-lem} 
For any hyperelliptic curve $C_0$, 
the obstruction pairing
$$\kappa:N\otimes \ker(\mu^d_C)\to \coker(\mu^d_C),$$
where $N$ is the normal space to the hyperelliptic locus at $C_0$
can be identified with the pairing $-\lan \cdot,\cdot\ran$
defined by \eqref{line-bun-obstr-pairing} for the canonical map $C_0\to \P^{g-1}$ and the line bundle $\LL=\OO_{\P^{g-1}}(d)$.
Furthermore, via the isomorphism $N\simeq H^0(\P^1,\OO(g-3))^\vee$ (see \eqref{obs-normal-space-eq}),
$-\kappa$ can be identified with the pairing
\begin{equation}\label{psi-d-eq}
\psi_d:H^0(\P^1,\OO(g-3))^\vee\ot H^0(\P^{g-1},\II_{\P^1}(d))\to H^0(\P^1,\OO((d-1)(g-1)-2))
\end{equation}
obtained from $\phi_d$ by dualization (see \eqref{phi-d-eq}).
\end{lem}

\Pf . The first assertion easily follows from formula \eqref{line-bun-obstr-pairing-formula}. The second follows
from this and from Corollary \ref{ker-obs-cor}.
\ed

\begin{prop}\label{Fitt-power-ideal-prop} 
The $0$th Fitting ideal of the sheaf $\coker(\mu^d)$ over $\MM_g$ coincides with
the ideal sheaf $\II_{\Hyp_g}^r$ where $r=(d-1)(g-1)-1$.
\end{prop}

\Pf . Since $\mu^d$ is non-surjective precisely on $\Hyp_g$,
we see that the radical of the $0$th Fitting 
ideal is equal to $\II_{\Hyp_g}$.
Furthermore, since the restriction of $\coker(\mu_d)$ to the hyperelliptic locus is a bundle of rank $r=(N-1)(g-1)-1$, 
by Lemma \ref{Fitt-power-ideal-lem}, the $0$th Fitting ideal is contained in $\II_{\Hyp_g}^r$, and 
working over a local ring $A$ in $\MM_g$ of a point $[C_0]$ in $\Hyp_g$, we can replace
the complex \eqref{mu-d-moduli-complex} by a quasi-isomorphic complex
$$\mu:V^0\to V^1$$ 
such that $\mu$ vanishes on the hyperelliptic locus. 

By Lemma \ref{two-obstr-lem}, the corresponding derivative map
$$\kappa=\nabla \mu: N\ot (V^0\otimes k) \to V^1\otimes k,$$
where $N$ is the normal space to the hyperelliptic locus, can be
identified (up to a sign) with the map $\psi_d$ (see \eqref{psi-d-eq}.
Thus, if $\MM_g$ has formal coordinates $x_1,\dots,x_n$ at $[C_0]$ and $\Hyp_g$ is cut out by 
$x_1,\dots,x_m$ then $\mu$ has the form
$$\mu=\phi_d(x_1,\ldots,x_m) \mod {\frak m}(x_1,\dots,x_m),$$
where ${\frak m}\sub A$ is the maximal ideal. Here we
identify $N^\vee$ with the space of linear forms in $x_1,\ldots,x_m$, and view $\phi_d$ as
a matrix of linear forms in $x_1,\ldots,x_m$, i.e., an element in $\Hom(V^0\otimes k,V^1\otimes k)\ot\lan x_1,\ldots,x_m\ran$.

We have to show that the maximal ($r\times r$) minors of 
the matrix $\mu$ generate 
$(x_1,\dots,x_m)^r$.  By Nakayama's Lemma, it is enough to check this modulo 
${\frak m}(x_1,\ldots,x_m)^r$, and thus we can replace $\mu$ by the matrix 
$\phi_d(x_1,\ldots,x_m)$.  

Assume first that $d\ge 3$. Then by Lemma \ref{ribbon-sur-lem}, the map
$\phi_d$ is surjective. This means that for appropriate choice of bases in $V^0\otimes k$ and $V^1\otimes k$,
the matrix of linear forms $\phi_d(x_1,\ldots,x_m)$ contains as a submatrix the $r\times rm$-matrix
$$\left(\begin{matrix} x_1\cdot 1_r & \ldots & x_m\cdot 1_r\end{matrix}\right).$$
Thus choosing $r$ columns suitably we can get any monomial in $x_1,\ldots,x_m$ as a maximal minor of $\phi_d$.

Now assume that $d=2$. In this case $m=r=g-2$, and $\phi_2$ is the canonical map
$$S_2(N^\vee)\to N^\vee \ot N^\vee.$$
Thus, the column of $\phi_2(x_1,\ldots,x_m)$ corresponding to a quadratic monomial $x_ix_j$, with $i\neq j$ is
$x_ie_j+x_je_i$, while the column corresponding to $x_i^2$ is $x_ie_i$.
Suppose we are given a monomial $x_{i_1}^{a_1}\ldots x_{i_p}^{a_p}$ of degree $r$, where $i_1<\ldots<i_p$, $a_i>0$.
Let us choose any partition
$$\{1,\ldots,m\}\setminus\{i_1,\ldots,i_p\}=S_1\sqcup\ldots\sqcup S_p,$$
with $|S_j|=a_j-1$, and take the columns corresponding to the following quadratic monomials:
$$x_{i_1}^2,(x_{i_1}x_j)_{j\in S_1},\ldots, x_{i_p}^2,(x_{i_p}x_j)_{j\in S_p}.$$
It is easy to see that the corresponding maximal minor of $\phi_2$ is equal to $\pm x_{i_1}^{a_1}\ldots x_{i_p}^{a_p}$.
\ed

\begin{rmk}
     The same calculation shows that for $0\le k\le r$, the $k$th Fitting 
ideal of $h^1(V^\cdot)$ is the $(r-k)$th power of the hyperelliptic ideal 
sheaf.  These are the Fitting ideals of a vector bundle of rank $r$ on the 
hyperelliptic substack, suggesting that $h^1(V^\cdot)$ is such a vector 
bundle, or equivalently (since the fibers are constant) that 
$h^1(V^\cdot)$ is scheme-theoretically supported on the hyperelliptic 
substack.  For $d>2$, this is clear from the structure of $\mu$:
$$(x_1,\dots,x_m)\coker(\mu)
\subset
(x_1,\dots,x_m){\frak m}\coker(\mu)
\subset
\bigcap_{n\ge 0} (x_1,\dots,x_m){\frak m}^n\coker(\mu)
=
0,$$
but this fails for $d=2$, $g\ge 4$, when the annihilator of $\coker(\mu)$ is 
actually contained in ${\frak m}^2$.
\end{rmk}

\begin{cor}\label{regular-map-cor} 
Let $\Bl_{\Hyp_g}\MM_g\to \MM_g$ denote the blow-up of the hyperelliptic locus. Then there is a closed embedding
$$\Bl_{\Hyp_g}\MM_g\hra \HH_g$$
over $\MM_g$, extending the regular map $\si:\MM_g\setminus \Hyp_g\to \HH_g$ corresponding to the image of the canonical 
embedding.
\end{cor}

\Pf . For sufficiently large $d$, we have a closed embedding of the relative Hilbert scheme into the relative Grassmannian
$$\HH_g\sub G_k(S^d p_*(\om_{\CC_g/\MM_g})).$$
Viewing $\si$ as a section of the relative Grassmannian, and applying Proposition \ref{Fitting-prop}, we
obtain that the closure of the image of $\si$ in the Grassmannian is equal to the blow-up of $\MM_g$
at the $0$th Fitting ideal of $\coker(\mu^d)$. By Proposition \ref{Fitt-power-ideal-prop}, it coincides
with the blow-up of the hyperelliptic locus $\Hyp_g\sub\MM_g$. Finally, since $\HH_g$ is closed in
the Grassmannian, we see that the closure of the image of $\si$ is contained in it (note that our blow-up is integral, hence, reduced).
\ed

\subsection{Embedding of hyperelliptic curves and ribbons into a weighted projective space}

Let us denote by $\RR_0$ the hyperelliptic ribbon (see Example \ref{hyperell-ribbon-ex}).

\begin{lem}
Let $C$ be a Gorenstein curve of genus g which is either a ribbon or
hyperelliptic (i.e., a possibly singular projective curve with a flat degree $2$ morphism to $\P^1$).  Then the moduli
stack of curves is unobstructed at $C$.
\end{lem}

\Pf .
We need to prove $\Ext^p_C(L_{C/k},\sO_C)=0$ for $p>1$.  In either 
case, there is an isotrivial family with generic fiber isomorphic to $C$ 
and the special fiber $\RR_0$ (for ribbons such a family is constructed in Lemma \ref{hyper-ribbon-deform-lem}).
Thus, by semicontinuity, it 
suffices to prove vanishing for $\RR_0$.
Note that $\RR_0$ is a divisor in the total space of the line bundle $\OO(-g-1)$ over $\P^1$, so we can view $\RR_0$ as a
divisor in the Hirzebruch surface $F_{g+1}$. Thus, we have a distinguished triangle
$$R\Hom_{\RR_0}(L_{{\RR_0}/F_{g+1}},\sO_{\RR_0})
\to
R\Hom_{\RR_0}(L_{\RR_0/k},\sO_{\RR_0})
\to
R\Hom_{F_{g+1}}(L_{F_{g+1}/k},\sO_{\RR_0})
\to\ldots$$
or equivalently
$$R\Hom_{\RR_0}(L_{\RR_0/k},\sO_{\RR_0})
\to
R\Gamma(T_{F_{g+1}/k}|_{\RR_0})
\to
R\Gamma(\sO_{\RR_0}(\RR_0))
\to\ldots$$
It thus suffices to show that
$$H^p(T_{F_{g+1}}|_{\RR_0})=0 \ \text{ for } p>1$$
and
$$H^p(\sO_{\RR_0}(\RR_0))=0 \text{ for } p>0.$$
The first follows immediately from the fact that $T_{F_{g+1}/k}$ is a vector bundle
and $\RR_0$ is $1$-dimensional, while the second follows by observing that
$$\sO_{\RR_0}(\RR_0)\cong \pi^*\sO_{\P^1}(2g+2),$$
where $\pi:\RR_0\to \P^1$ is the double cover,
so $\pi_*(\sO_{\RR_0}(\RR_0))\cong \OO_{\P^1}(2g+2)\oplus \OO_{\P^1}(g+1)$.
\ed

Let $X_g$ denote the weighted projective space with $g$ homogeneous coordinates of degree $1$, $u_0,\ldots,u_{g-1}$, 
and $g-2$ coordinates of degree $2$, $v_0,\ldots,v_{g-3}$.  

\begin{lem}\label{wt-proj-emb-lem} 
Let $C$ be a projective curve of genus $g\ge 3$, which is either a ribbon or a smooth curve.
Then the codimension of the multiplication map $\mu^2_C$ is at most $g-2$, so we can choose
$g-2$ elements $v_0,\ldots,v_{g-3}$ of $H^0(C,\omega_C^{\ot 2})$ spanning $\coker(\mu^2_C)$.
Let $u_0,\ldots,u_{g-1}$ be a basis of $H^0(C,\om_C)$. Then $(u_0,\ldots,u_{g-1},v_0,\ldots,v_{g-3})$
give a closed embedding of $C$ into $X_g$.
\end{lem} 
  
\Pf . If $C$ is smooth nonhyperelliptic or a nonhyperelliptic ribbon, 
then already $u_0,\ldots,u_{g-1}$ give the usual canonical embedding of $C$.
 
Assume that $C$ is smooth hyperelliptic of genus $\ge 3$. 
Then we just need to see that sections of $\om_C^2$ separate points and tangent vectors, but this follows from
the fact that $\deg(\om_C^2(-p_1-p_2))=4g-6>2g-2$.

Now assume that $C=\RR_0$, the hyperelliptic ribbon. Then the desired embedding is the composition of the natural embedding
of $C$ into the Hirzebruch surface $F_{g+1}$ with the embedding $F_{g+1}\sub X_g$. More explicitly, $\RR_0$ is
given in $X_g$ by the homogeneous equations
\begin{equation}\label{hyperell-ribbon-equations}
\begin{array}{l}
u_iu_j=u_ku_l, \ \ u_iv_j=u_kv_l, \text{ for } i+j=k+l,\\
v_iv_j=0, \text{ for all } i,j.
\end{array}
\end{equation}
where the first two sets of equations cut out $F_{g+1}$.
\ed  

\begin{lem}\label{Koszul-lem}
The homogeneous coordinate algebra $A_g$ of $\RR_0$ in $X_g$
is given by the quadratic relations \eqref{hyperell-ribbon-equations}. This algebra 
is Koszul with respect to the grading $\deg(u_i)=\deg(v_j)=1$. In particular, the syzygies between 
relations \eqref{hyperell-ribbon-equations}, denoted as $(uu)_0$, $(uv)_0$ and $(vv)_0$,
are generated by linear syzygies of the schematic form
\begin{equation} \label{hyperell-ribbon-syzygies-eq}
\begin{array}{l}
u(uu)_0, \\ 
v(uu)_0+u(uv)_0, \\
v(uv)_0+u(vv)_0, \\
v(vv)_0,
\end{array}
\end{equation}
(where $u\cdot (uu)_0$ denote linear combinations with coefficients linear 
in $u$ of the relations of the type $(uu)_0$, etc.).
\end{lem}

\Pf . We have 
$$\OO_{\RR_0}=\OO_{\P^1}\oplus \OO_{\P^1}(-g-1), \ \ \om_{\RR_0}^{\ot
  n}=\OO_{\P^1}(n(g-1))\oplus\OO_{\P^1}(n(g-1)-g-1),$$ and the embedding
into $X_g$ corresponds to the standard monomial basis basis $(u_i)$ of
$H^0(\RR_0,\om_{\RR_0})\simeq H^0(\P^1,\OO_{\P^1}(g-1))$, and the standard
basis $(v_j)$ of $H^0(\P^1,\OO_{\P^1}(g-3))\sub
H^0(\RR_0,\om_{\RR_0})$. This easily implies that the polynomial algebra
$k[u,v]$ surjects onto the canonical ring of $\RR_0$, $\bigoplus_{n\ge
  0}H^0(\RR_0,\om_{\RR_0}^{\ot n})$, so the algebra $A_g$ can be identified
with this canonical ring, which has Hilbert series $1+gt+(g-2)\sum_{n\ge
  2}(2n-1)t^n$.

Let $\wt{A}_g$ be the quadratic algebra defined by the relations \eqref{hyperell-ribbon-equations}. 
We claim that these relations form a quadratic Gr\"obner basis with
respect to the order $u_0<\ldots<u_{g-1}<v_0<\ldots<v_{g-3}$, and that the natural surjective homomorphism
$$\wt{A}_g\to A_g$$ is an isomorphism. Indeed, the normal quadratic
monomials have form $u_iu_{i+1}$ and $u_iv_0$.  This gives a set of normal
monomials matching the Hilbert series of $A_g$ (where we use the grading
$\deg(u_i)=1$, $\deg(v_j)=2$). Hence, the normal monomials project to a
basis of $A_g$. It follows that they form a basis in $\wt{A}_g$ and that
$\wt{A}_g\simeq A_g$.  By a standard criterion, this implies that the
algebra $A_g$ is Koszul.

It follows the module of syzygies is generated by linear syzygies. 
It is easy to check that the linear syzygies have the required form.
\ed
  
\begin{definition} 
Let us say that a subscheme of $X_g$ is a {\it ribbon in canonical form} if
it is given by homogeneous equations of the form
\begin{equation}\label{can-ribbon-equations}
\begin{array}{l}
u_iu_j-u_ku_l=\ell_{i,j,k}(v), \text{ for } i+j=k+l,\\ 
u_iv_j=u_kv_l, \text{ for } i+j=k+l,\\
v_iv_j=0, \text{ for all } i,j,
\end{array}
\end{equation}
where $\ell_{i,j,k}(v)$ are some linear forms in $v$,
and if this subscheme has the same Hilbert series as the hyperelliptic ribbon.
\end{definition}

\begin{lem}\label{ribbon-eq-syz-lem} 
(i) Let $C$ be a ribbon of genus $g$.
For appropriate choice of a basis $(u_i)$ of $H^0(C,\om_C)$ and
a choice of $g-2$ elements $v_0,\ldots,v_{g-3}$ of $H^0(C,\omega_C^{\ot 2})$ spanning $\coker(\mu^2_C)$,
the image of $C$ in $X_g$ will be a ribbon in canonical form.
Conversely, every ribbon in canonical form is a ribbon of genus $g$.

\noindent
(ii) Let $(uu)$ be the first group of equations
\eqref{can-ribbon-equations} (with quadratic parts $(uu)_0$).  Then the
module of syzygies between equations \eqref{can-ribbon-equations} is
generated by the syzygies of the form
\begin{align*} 
& u(uu) +(uv)_0, \\
& v(uu)+u(uv)_0+(vv)_0, \\
& v(uv)_0+u(vv)_0, \\
& v(vv)_0.
\end{align*}
Furthermore the linear parts of these syzygies are exactly all the syzygies
between the relations \eqref{hyperell-ribbon-equations}.

\noindent
(iii) Assume the characteristic is $\neq 2$. Then a hyperelliptic curve $C$ of genus $g$ can be embedded
into $X_g$, so that the corresponding subscheme of $X_g$ is given by the following equations deforming
\eqref{hyperell-ribbon-equations}:
\begin{equation}\label{hyperell-equations}
\begin{array}{l}
u_iu_j=u_ku_l, \ \ u_iv_j=u_kv_l, \text{ for } i+j=k+l,\\
v_iv_j=p_{ij}(u), \text{ for all } i,j,
\end{array}
\end{equation}
for some homogeneous polynomials of degree $4$, $p_{ij}(u)$.
Furthermore, all syzygies between these equations have schematic form
\begin{align*} 
& u(uu)_0,\\
& u(uv)_0+v(uu)_0,\\
& u(vv)+v(uv)_0+uuu(uu)_0,\\
& v(vv)+uuu(uv)_0+uuv(uu)_0,
\end{align*}
and the linear parts of these syzygies are exactly all the syzygies between the relations
\eqref{hyperell-ribbon-equations}.
\end{lem}

\Pf . (i) Let $C$ be a ribbon. We have an exact sequence
$$0\to \OO_{\P^1}(-2)\to \om_C\to \OO_{\P^1}(g-1)\to 0$$
inducing an isomorphism $H^0(C,\om_C)\rTo{\sim}H^0(\P^1,\OO_{\P^1}(g-1))$.
We choose a basis $(u_i)$ of $H^0(C,\om_C)$ corresponding to the standard monomial basis
of $H^0(\P^1,\OO_{\P^1}(g-1))$.
Next, the exact sequence
$$0\to \OO_{\P^1}(g-3)\to \om_C^{\ot 2}\to \OO_{\P^1}(2g-2)\to 0$$
gives an exact sequence
$$0\to H^0(\P^1,\OO_{\P^1}(g-3))\to H^0(\om_C^{\ot 2})\to H^0(\P^1,\OO_{\P^1}(2g-2))\to 0$$
and we choose $(v_j)$ to come from the standard monomial basis of $H^0(\P^1,\OO_{\P^1}(g-3))$.
Then the relations $(uv)_0$ and $(vv)_0$ are satisfied. Furthermore,
$(uu)_0$ is satisfied modulo the ideal generated by $(v_j)$, so we get some equations of
the form \eqref{can-ribbon-equations}. Let $C_{eq}$ be the subscheme defined by these equations.
Then $C\sub C_{eq}$ and the Hilbert polynomial of $C_{eq}$ is bounded above termwise by the Hilbert polynomial
of $\RR_0$, which is equal to that of $C$. Hence, $C=C_{eq}$. 

Conversely, assume $C$ is a subscheme of $X_g$ given by equations \eqref{can-ribbon-equations}.
Then $\II_C$ is sandwiched 
between $\II_{C_0}$ and $\II_{C_0}^2$, where $\II_{C_0}$ is the ideal generated by
$v_0,\dots,v_{g-3}$ and the $(uu)_0$ equations.  This implies (since $C_0$ is smooth) 
that $C_0$ is the singular subscheme of $C$, and that we have an exact sequence
$$0\to \LL\to \OO_C\to \OO_{C_0}\to 0$$
where $\LL$ is scheme-theoretically supported on $C_0\simeq \P^1$.
Now we observe that rescaling the variables $(v_j)$ we get a family over $\A^1$, with the fiber
$C$ over any point of $\A^1\setminus \{0\}$ and with the fiber $\RR_0$ over $0$.
Since the Hilbert series does not change, this is a flat family. Hence, $\LL$ is a flat deformation of
$\OO_{\P^1}(-g-1)$ on $\P^1$, so $\LL\simeq \OO_{\P^1}(-g-1)$, and $C$ is a ribbon.

\noindent
(ii) As in part (i), rescaling the $(v_j)$ variables, we can view a ribbon in canonical form as a flat $1$-parameter deformation $C$
of the hyperelliptic ribbon $\RR_0$. 
The relations \eqref{can-ribbon-equations} correspond to generators of the ideal sheaf $\II_C$ of degrees $2$, $3$ and $4$,
so that we have a surjection of the form
$$R_2\ot \sO(-2)\oplus R_3\ot \sO(-3)\oplus R_4\ot \sO(-4)\to \II_C.$$
on $X_g$. Let  $\KK_C$ denote the kernel of this map.
The sheaves $\KK_C$ form a flat family as we degenerate $C$ into $\RR_0$.
Furthermore, since $H^{>0}(X_g,\II_{\RR_0}(m))=0$ for $m\ge 2$, 
the spaces of syzygies $H^0(\KK_C(m))$ for $m\ge 3$ also form a flat family. 
Thus, the syzygies \eqref{hyperell-ribbon-syzygies-eq} extend to sections
of $\KK_C(m)$, $3\le m\le 6$, generating $\KK_C$. They automatically have the required form,
and their linear parts specialize to \eqref{hyperell-ribbon-syzygies-eq}. 

\noindent
(iii) Since the characteristic is $\neq 2$, $C$ can be embedded as $y^2=h$
in the total space of the line bundle $\OO_{\P^1}(g+1)$ over $\P^1$, which
is an open subset of the Hirzebruch surface $F_{g+1}$ cut out by equations
$(uu)_0$ and $(uv)_0$ in $X_g$. It is easy to see that the equation $y^2=h$
corresponds to the remaining equations \eqref{hyperell-equations} (of the
form $(vv)$) in $X_g$.  The statement about syzygies is proved similarly to
part (ii). Note that the form of the syzygies is as written since they
should be eigenvectors for the operator of negating all $v_i$'s.  \ed

\subsection{Ribbon locus in the Hilbert scheme}

We want to describe the locus of canonically embedded ribbons as a locally closed subscheme
of $H_g(\P^{g-1})$.

First, we need a characterization of rational normal curves in the corresponding Hilbert scheme.
Let $H^{rnc}(\P^d)$ denote the Hilbert scheme of curves in $\P^d$ with the Hilbert polynomial
$P(n)=nd+1$, so that a rational normal curve in $\P^d$ defines a point of $H^{rnc}(\P^d)$.

\begin{lem}
Let $[C]\in H^{rnc}(\P^d)$ be a point such that $C$ is smooth and nondegenerate,
i.e., $H^0(\P^d,\II_C(-1))=0$. Then $C$ is a rational normal curve in $\P^d$.
\end{lem}

\Pf . Since $C$ has arithmetic genus $0$, we should have $C=C_1\sqcup\ldots\sqcup C_k$,
where each $C_i$ is isomorphic to $\P^1$. Let $\deg(\OO_{\P^d}(1)|_{C_i})=d_i$. Then we have
$$P(n)=nd+1=(nd_1+1)+\ldots+(nd_k+1)=n(d_1+\ldots+d_k)+k,$$ 
so $k=1$. Thus, $C$ is irreducible. But it is well known that an irreducible nondegenerate curve of degree $d$
in $\P^d$ is a rational normal curve.
\ed

Thus, the locus $RNC\sub H^{rnc}(\P^d)$ of rational normal curves can be described as an open subset of
smooth and nondegenerate curves.

Now let us consider the nested Hilbert scheme $H^{\nest}(\P^{g-1})$ of pairs $C_0\sub C\sub \P^{g-1}$,
where $C$ is a point of $H_g(\P^{g-1})$ and $C_0$ is a point of $H^{rnc}(\P^{g-1})$.
Let us consider the closed subscheme
$$Z^{\nest}\sub H^{\nest}(\P^{g-1})$$ defined by the condition that
$\II_{C_0}^2\sub \II_C$ (or equivalently, that the composition
$\II_{C_0}^2\to \OO_{\P^{g-1}}\to \OO_C$ vanish).  Let us consider the
natural projection
\begin{equation}\label{nested-Hilb-proj-eq}
\pi:Z^{\nest}\to H_g(\P^{g-1}).
\end{equation}
Let $B\sub Z^{\nest}$ be the closed subset of $(C_0,C)$ such that $C_0$ is not a rational normal curve.


\begin{lem}\label{nested-Hilb-lem}
(i) Let $(C_0,C)$ be a point of $Z^{\nest}$ such that $C_0$ is a rational normal curve. Then $C$ is a canonical ribbon
$\RR_\la$ for some $\la\in \P H^0(C_0,\om_{C_0}(1))^*$.

\noindent
(ii) The natural projection $Z^{\nest}\setminus B\to RNC$ is the
$\P^{g-3}$-bundle associated with the vector bundle with fiber
$H^0(C_0,\om_{C_0}(1))^\vee$ over $C_0$.
\end{lem}

\Pf . (i) The exact sequence
$$0\to \II_{C_0}/\II_C\to \OO_C\to \OO_{C_0}\to 0$$ shows that
$\II_{C_0}/\II_C$ has Hilbert polynomial $n\mapsto n(g-1)-g$.  Furthermore,
this is a quotient of $\II_{C_0}/\II_{C_0}^2\simeq \om_{C_0}(-1)^{\oplus
  g-2}$. It follows that $\II_{C_0}/\II_C$ is a sheaf of rank $1$ on
$\P^1$. But a sheaf of rank $1$ on $\P^1$ has form $\OO_{\P^1}(a)\oplus
\TT$, where $\TT$ is torsion, so its Hilbert polynomial with respect to
$\OO_{\P^1}(g-1)$ is $n(g-1)+a+\ell(\TT)+1$. Thus, we should have
$a+\ell(\TT)+1=-g$, so $a=-g-1-\ell(\TT)$. But if $a<-g-1$ then there are
no morphisms from $\II_{C_0}/\II_{C_0}^2$ to $\OO_{\P^1}(a)$, so we should
have $\TT=0$ and $\II_{C_0}/\II_C\simeq \om_{C_0}(-1)$.

\noindent
(ii) As in (i) we see that $Z^{\nest}\setminus B$ is isomorphic to the
relative Quot-scheme corresponding to quotients of $\II_{C_0}/\II_{C_0}^2$
with the Hilbert polynomial $n(g-1)-g$. We have a natural map from the
projective bundle in question to the Quot-scheme, which is bijective on
geometric points. It remains to observe that the Quot-scheme is smooth over
$RNC$. Indeed, let $L=\om_{C_0}(-1)$. Then for a surjective map
$\la:L^{\oplus g-2}\to L$ we have $\ker(\la)\simeq L^{\oplus g-3}$ so
$\Hom(\ker(\la),L)$ has dimension $g-3$.  \ed

\begin{prop}\label{nest-ribbon-prop} 
Assume the characteristic of $k$ is $\neq 2$.
Then the projection $\pi$ induces a closed embedding $Z^{\nest}\setminus B\to H_g(\P^{g-1})\setminus \pi(B)$.
\end{prop}

\Pf . The main observation is that if $(C_0,C)\in (Z^{\nest}\setminus
B)(R)$ then $C_0(R)$ can be recovered from $C(R)$ as its singular locus,
i.e., as the subscheme defined by the $1$st Fitting ideal of $\Om_{C/R}$
(here $R$ is a local commutative ring). Indeed, the complement to an
$R$-point of $C_0$ can be identified with $\Spec(R[x])$ so that the
corresponding open subset in $C$ is $\Spec(R[x,\epsilon]/\epsilon^2)$. Then
we claim that the singular subscheme is cut out by the ideal $(2\epsilon)$.
Indeed, $\Omega_{C/R}$ is locally generated by $dx$ and $d\epsilon$ subject
to the relation $2\epsilon d\epsilon=0$, so we get that the Fitting ideal
is generated by $2\epsilon$.

Note also that if $(C_0,C)\in Z^{\nest}\setminus B$ then we cannot have $C\in \pi(B)$:
any subscheme $C'_0\sub C$ with $\II_{C'_0}^2\sub \II_C$ should have $C_0$ as its reduced subscheme, so
$C_0\sub C'_0$. Since the Hilbert schemes of $C_0$ and $C'_0$ are the same, this implies that $C'_0=C_0$.

Thus, we have $\pi^{-1}(\pi(B))=B$, and so the morphism
$$\pi: Z^{\nest}\setminus B\to H_g(\P^{g-1})\setminus \pi(B)$$ is
proper. Since it induces a bijection on $R$-points, for $R$ local, it is a
closed embedding.  \ed

\begin{definition} Assume the characteristic of $k$ is $\neq 2$. The {\it ribbon locus} in the Hilbert scheme, 
$$R_g\sub H_g(\P^{g-1})$$ 
is the image of the above locally closed embedding of $Z^{\nest}\setminus B$.
\end{definition}

\begin{rmk}
There are difficulties in defining the ribbon locus in characteristic 
$2$, coming from the fact that the map from $Z^{\nest}\setminus B$, though a
bijection, is no longer an embedding (it is inseparable over its image). 
As a result, any attempt to extend this definition to include 
characteristic 2 will either fail to be flat or fail to have reduced 
fiber over $2$.
\end{rmk}

\subsection{The blow-up of the ribbon locus}

Let $\wt{\MM}_g$ be the $\PGL_g$-bundle over $\MM_g$ associated with the Hodge vector bundle $\pi_*(\om_{\CC_g/\MM_g})$,
and let $\wt{\Hyp}_g\sub \wt{\MM}_g$ be the preimage of the hyperelliptic locus.
By Corollary \ref{regular-map-cor}, we have a regular map
\begin{equation}\label{blow-up-map}
\Bl_{\wt{\Hyp}_g}\wt{\MM}_g\to H_g(\P^{g-1}),
\end{equation}
where $H_g(\P^{g-1})$ is the Hilbert scheme of canonical curves of genus $g$ in $\P^{g-1}$.

On the other hand, we can consider
the blowup of the Hilbert scheme in the closure of the ribbon subscheme $\ov{R}_g\sub H_g(\P^{g-1})$.

It is known that $R_g$ is contained in the smooth locus of $H_g(\P^{g-1})$ (see \cite[Sec.\ 6]{BE}).
The tangent space to $H_g(\P^{g-1})$ at a ribbon $C\sub\P^{g-1}$ is $H^0(C,N_C)$, where $N_C$ is the normal sheaf. 
Let $D=C^{red}\simeq \P^1$. Then $h^0(N_C)=g^2+3g-4$ and there are natural surjective morphisms 
$$H^0(C,N_C)\to H^0(D,N_C|_D)\to H^0(\P^1,\OO(2g+2)).$$
Indeed, this follows from the exact sequences
$$0\to \OO_{\P^1}(g+1)^{g-3}\to N_C|_D\to \OO_{\P^1}(2g+2)\to 0$$
$$0\to N_C|_D\ot \OO_{\P^1}(-g-1)\to N_C\to N_C|_D\to 0$$
(see \cite[Sec.\ 6]{BE}).
Note that the first sequence here is dual to the exact sequence
$$0\to \II_D^2/\II_D\II_C\to \II_C/\II_D\II_C\to \II_C/\II_D^2\to 0.$$


Let us denote by 
$$N_{R_g}H_g(\P^{g-1}):= N_{\ov{R}_g} H_g(\P^{g-1})|_{R_g}$$ 
the normal bundle of $\ov{R}_g$ in the Hilbert scheme,
over $R_g\sub \ov{R}_g$.

\begin{lem}\label{normal-ribbon-lem} 
Assume that the characteristic is $\neq 2$.
Then the above construction defines an isomorphism
\begin{equation}\label{discriminant-map-eq}
N_{R_g} H_g(\P^{g-1})|_{[C]}\rTo{\sim} H^0(\P^1,\OO(2g+2)).
\end{equation}
\end{lem}

\Pf .  Recall that by Lemma \ref{nested-Hilb-lem}(ii) and Proposition
\ref{nest-ribbon-prop}, the ribbon locus $R_g$ can be identified with a
$\P^{g-3}$-bundle over the locus $RNC$ of rational normal curves.  Further,
we can identify $RNC$ with the homogeneous space $\PGL_g/\PGL_2$.

We thus find that the tangent 
sheaf to the ribbon subscheme of the Hilbert scheme fits into a short 
exact sequence with fibers
$$0\to \gl_2\times \gl_1
 \to \gl_g\oplus \Gamma(\sO_{\P^1}(g-3))^*
 \to T
 \to 0.$$
This has the correct dimension $g^2+(g-2)-5 = (g^2-1)+(3g-3)-(2g+3)$ to be 
the kernel of the map to $\Gamma(\sO_{\P^1}(2g+2))$, and thus it remains 
only to show that it maps to said kernel, or equivalently that the 
induced map from $\gl_g\oplus \Gamma(\sO_{\P^1}(g-3))^*$ to 
$\Gamma(\sO_{\P^1}(2g+2))$ is zero.  An element of $\gl_g$ corresponds to a 
degree 0 derivation of the homogeneous coordinate ring of $\P^{g-1}$, and 
the induced map from $I_C$ to $\sO_D$ is obtained by applying the derivation 
then restricting to $\sO_D$.  Since any element of $\gl_g$ takes $I_D^2$ to 
$I_D$, the map $I_C\to \sO_D$ corresponding to such an element vanishes on 
$I_D^2$ as required.  Similarly, the action of the additive group 
$\Gamma(\sO_{\P^1}(g-3))^*$ maps $I_C$ to $I_D$ (it takes polynomials 
vanishing on $C$ to polynomials vanishing on a different ribbon over the 
same base curve $D$), and thus any element of the Lie algebra induces the 
trivial map $I_C\to \sO_D$.

Thus, we have a well defined surjective map \eqref{discriminant-map-eq}. Since the dimensions
of both spaces are equal to $2g+3$, we deduce that it is an isomorphism.
\ed

We define the discriminant locus $\DD\sub \P N_{R_g} H_g(\P^{g-1})$
to be the divisor corresponding under the map \eqref{discriminant-map-eq} to the locus of 
sections of $\OO_{\P^1}(2g+2)$ with non-simple zeros. 

\begin{thm}\label{two-blow-ups-thm}
Assume that the characteristic is $\neq 2$.
Then the morphism of the coarse moduli obtained from \eqref{blow-up-map} factors through an open immersion
$$[\Bl_{\wt{\Hyp}_g}\wt{\MM}_g]^{\coarse}\hra \Bl_{\ov{R}_g}H_g(\P^{g-1})$$
Furthermore, its image is the complement to the union of 
\begin{itemize}
\item the preimage of $\ov{R}_g\setminus R_g$,
\item the strict transform of the singular
locus in $H_g(\P^{g-1})$, 
\item and of the closure $\ov{\DD}$ of the discriminant locus in the exceptional locus of the blow-up.
\end{itemize}
\end{thm}


\begin{rmk}
Away from the hyperelliptic locus the stacky structure on $\wt{\MM}_g$ is trivial, whereas along
$\wt{\Hyp}_g$, we have automorphism group $\Z/2$, coming from the hyperelliptic
involution. This $\Z/2$ acts trivially on the tangent space to $\wt{\Hyp}_g$ and acts by $-1$
on the normal space to $\wt{\Hyp}_g$ in $\wt{\MM}_g$. This easily implies that 
$[\Bl_{\wt{\Hyp}_g}\wt{\MM}_g]^{\coarse}$ is smooth and coincides with the blow up of
the coarse moduli space $[\wt{\MM}_g]^{\coarse}$ along $[\wt{\Hyp}_g]^{\coarse}$.
\end{rmk}

\begin{definition}
Suppose $X$ is a smooth scheme, $Z\sub X$ is a smooth subscheme.
Assume we are given a dvr $R$ with residue field $k$ and fraction field $K$,
and a map $t:\Spec(R)\to X$ such that $t(\Spec(K))\in X\setminus Z$ and $t(\Spec(k))=z\in Z$.
Then there is a unique lifting of $t$,
$$\wt{t}:\Spec(R)\to \Bl_Z(X),$$ so that $\wt{t}(\Spec(k))$ is a point in
the exceptional divisor over $z\in Z$, and thus corresponds to a normal
direction to $Z$ at $z$. We will refer to $\wt{t}(\Spec(k))$ as the {\it
  normal vector associated with} $t:\Spec(R)\to X$.
\end{definition}





First, we need a characterization of normal vectors to the ribbon locus (resp., hyperelliptic locus)
associated with deformations of ribbons (resp., hyperelliptic curves) embedded into $X_g$.
Below we always assume that the characteristic is $\neq 2$.

\begin{lem}\label{normal-ribbon-equation-lem} 
(i) Let $R$ be a dvr with the maximal ideal $(\pi)$, $d\ge 1$ an integer, and let
\begin{align*}
&(uu)_0+(v)+\pi^d(uu)=0,\\
&(uv)_0+\pi^d(uuu)+\pi^d(uv)=0,\\
&v_iv_j-\pi^d p_{ij}(u)+\pi^d(uuv)+\pi^d(vv)
\end{align*}
be a subscheme in $X_g\times R/(\pi^{d+1})$ 
giving a deformation of a nonhyperelliptic ribbon $C$ over $R/(\pi^d)$.
Then the polynomials $(p_{ij}(u))$ correspond to the section of 
$\sO_{\P^1}(2g+2)$ associated with the normal vector to the ribbon locus defined by this deformation
(see \ref{normal-ribbon-lem}).

\noindent
(ii) Similarly, let 
\begin{align*}
&(uu)_0+\pi^d(v)+\pi^d(uu)=0,\\
&(uv)_0+\pi^d(uuu)+\pi^d(uv)=0,\\
&v_iv_j-p_{ij}(u)+\pi^d(uuv)+\pi^d(vv)+\pi^d(uuuu)
\end{align*}
be a subscheme in $X_g\times R/(\pi^{d+1})$ 
giving a deformation of a hyperelliptic curve $C$ over $R/(\pi^d)$.
Then the terms $(v)$ define a section of $N(g+1)$ on the rational normal curve $\P^1\simeq D\sub \P^{g-1}$, where $N$
is the normal bundle to $D$ in $\P^{g-1}$, and this section determines the normal vector to the hyperelliptic
locus associated with this deformation.
\end{lem}

\Pf . (i) Let $\II_C$ denote the ideal of $C$ in $\P^{g-1}$, and let $D\sub C$
be the reduced subscheme, which is a rational normal curve in $\P^{g-1}$
with ideal $\II_D$.  Then the section of $\sO_{\P^1}(2g+2)$ is obtained
from a tangent vector to $C$ as a point of the Hilbert scheme of $\P^{g-1}$
by the composition of natural maps,
\begin{align*}
&\Hom(\II_C,\OO_C)\to \Hom(\II_C,\OO_D)\simeq\Hom(\II_C/\II_C\II_D,\OO_D)\to \Hom(\II_D^2/\II_C\II_D,\OO_D)\\
&\simeq \Hom(\sO_{\P^1}(-2g-2),\sO_{\P^1}).
\end{align*}
Since our first order deformation is given by equations in $X_g$, we have in addition to compose the above map
with the projection
$$\Hom(\II_{C,X_g},\OO_C)\to\Hom(\II_C,\OO_C),$$
where $\II_{C,X_g}$ is the ideal of $C$ in $X_g$.
It remains to observe that $(v_iv_j)$ is precisely the image of $\II_D^2$ in $\II_{C,X_g}$.

\noindent
(ii) The fact that the terms $(v)$ define a morphism $\II_D\to \OO_{\P^1}(g+1)$ is dictated by the 
syzygies of the $(uu)_0$ relations. The codimension of the hyperelliptic locus is equal to $g-2$, as
is the dimension of $H^0(\P^1,N(g+1))$ (as follows from Lemma \ref{normal-bundle-lem}).
Hence, it remains to check that nontrivial terms $(v)$ lead to a nonhyperelliptic deformation.
Indeed, if the deformation were hyperelliptic then by Lemma \ref{ribbon-eq-syz-lem}(iii), we could transform
our equations to the canonical form \eqref{hyperell-equations} by an automorphism of $X_g$ over $R/(\pi^{d+1})$,
trivial modulo $\pi^d$. But such automorphisms cannot change the $(v)$ terms in the first set of equations.
\ed

We use the following terminology below. Given a deformation functor $F$, an Artinian local ring $B$
and an element $\eta\in F(B)$ we define the tangent space to $F$ at $B$, $T_{F,\eta}$ as the preimage of
$\eta$ under the map $F(B[\eps]/(\eps^2))\to F(B)$. It has a natural structure of a $B$-module.
Given a square zero extension 
$$0\to M\to A\to B\to 0,$$
where $M$ is a free $B$-module,
and an element $\eta\in F(B)$, there is a natural transitive action of $T_{F,\eta}\ot_B M$ on
the preimage of $\eta$ in $F(A)$, coming from an isomorphism of algebras
$$A\times_B B[\eps_1]/(\eps_1^2)\times_B\ldots\times_B B[\eps_n]/(\eps_n^2)\simeq A\times_B A$$
(where $M\simeq B^{\oplus n}$). If $F$ is prorepresentable then this action is simply transitive.
Thus, a choice of a point in $F(A)$ over $\eta$ endows the space of all liftings of $\eta$ to $F(A)$ 
with a structure of a $B$-module.

\begin{prop}\label{hyperell-to-ribbon-prop}
Let $R$ be a dvr with residue characteristic not $2$, and let 
$C/R$ be a smooth curve with hyperelliptic special fiber and 
nonhyperelliptic general fiber.  Let $d$ be the largest integer such that 
$C_{R/{\frak m}^d}$ is hyperelliptic (i.e., such that there exists a 
hyperelliptic curve over $R$ agreeing with $C$ modulo ${\frak m}^d$).  Then 
the corresponding family $C'$ in the canonical Hilbert scheme is a ribbon to 
order precisely $2d$. Furthermore, $C'\mod{\frak m}$ is determined 
by the normal vector to the hyperelliptic locus associated with
$C\mod {\frak m}^{d+1}$, while
the normal vector to the ribbon locus associated with $C'\mod {\frak m}^{2d+1}$ is determined
by $C\mod {\frak m}$.
\end{prop}

\Pf . We know that the family $C'/R$ in the canonical Hilbert scheme will have a ribbon as a central fiber
(see Lemma \ref{limit-lem}). Furthermore, by choosing extra generators in $H^0(\om_{C'}^{\ot 2})$ we
can lift $C'$ to a family in the Hilbert scheme of $X_g$. Thus, for the rest of the proof we will study this family
of subschemes of $X_g$.

Let $C_h/R$ be a hyperelliptic curve approximating $C$ to order $d$, 
so that we may view $C$ as a deformation of $C_h \mod \fm^d$. Recall that the set of extensions of
$C_h \mod \fm^d$ to a curve over $R/\fm^{2d}$ can be identified with the $R/\fm^d$-module, $T_{\MM,C_h\mod \fm^d}$,
the tangent space to $C_h\mod \fm^d$.  

By Lemma \ref{ribbon-eq-syz-lem}, $C_h$ is given in $X_g$ by equations
\eqref{hyperell-equations} for some homogeneous polynomials of degree $4$,
$p_{ij}(u)$.  Since $C$ is a flat deformation of $C_h\mod \fm^d$, the
equations of $C$ are obtained from those of $C_h$ by adding polynomials
with coefficients in ${\frak m}^d$.  Now, the hyperelliptic involution of
$C_h$ extends to an action on $X_g$ by negating the $v$ variables, and
pulling back through this involution in general changes the equations for
$C$, taking each generator $p$ to $\pm \iota(p)$ where $\iota$ negates the
$v$ variables and the sign is chosen to preserve the corresponding equation
of $C_h$.  Then we may define a new deformation $C_e\mod \fm^{2d}$ of
$C_h\mod \fm^d$ that has equations $(p\pm \iota(p))/2$; linearity of the
tangent space tells us that this is still a well-defined flat deformation
of $C_h$ over $R/{\frak m}^{2d}$. In other words, we split the tangent
vector to $C_h\mod \fm^d$ corresponding to $C\mod \fm^{2d}$ into even and
odd parts with respect to the involution, and take $C_e\mod \fm^{2d}$ to be
the deformation corresponding to the even part.

We then have the following analysis of the equations for $C_e\mod \fm^{2d}$.
Note that the new terms of these equations should define a morphism in $\Hom(\II_{C_h},\OO_{C_h})$,
so applying the syzygies of $C_h$ to the equations for $C_e\mod \fm^{2d}$ should give us elements
of $\II_{C_h}$.

(a) The $(uu)$ equations for $C_e\mod \fm^{2d}$ have only quadratic terms in $u$
(since the $v$ terms are killed by symmetrization), and applying to them the syzygies of the form $u(uu)_0$
should give us elements of the ideal of $C_h$, which therefore belong to the ideal generated by $(uu)_0$.
Hence, the $(uu)$ equations for $C_e\mod \fm^{2d}$
give a flat deformation of the underlying rational 
normal curve, and thus themselves cut out a rational normal curve. We may thus act by 
$\gl_g\otimes {\frak m}^d/{\frak m}^{2d}$ (and a corresponding change of 
basis on the space of equations) to make this deformation trivial. 
(This also commutes with $\iota$, so changes $C$ in a compatible way.)

(b) The $(uv)$ equations for $C_e\mod \fm^{2d}$ have only terms of the form $uv$ (the terms cubic in $u$ are 
eliminated by symmetrization), so looking at the linear syzygies between $(uu)_0$ and $(uv)_0$, 
given that we have normalized as in (a), we obtain that these relations define a flat deformation of 
the total space of $\sO_{\P^1}(g+1)$ 
over the normal rational curve embedded into $X_g$. It follows that 
we can similarly eliminate that deformation by acting by 
$\gl_{g-2}\otimes {\frak m}^d/{\frak m}^{2d}$.

(c) The $vv$ equations deform by adding terms that are quadratic in $v$ and 
terms that are quartic in $u$.  Since every quadratic term already appears 
in a generator of $C_h$, we can eliminate the $vv$ terms by a change of 
basis.  But then the resulting equations of $C_e\mod \fm^{2d}$ are of precisely the 
same form as those of $C_h$, and we can in fact extend $C_e\mod \fm^{2d}$ to a 
hyperelliptic curve $C'_h$ over $R$.

Thus, replacing $C_h$ with $C'_h$, without loss of generality we may 
assume that $C\mod \fm^{2d}$ is an odd deformation of $C_h\mod \fm^{2d}$ (with respect to
the hyperelliptic involution); that is, the 
equations of $C\mod {\frak m}^{2d}$ are obtained by:

(a) adding $v$ terms with coefficients in ${\frak m}^d$ to the $(uu)_0$ equations;

(b) adding $uuu$ terms (ditto) to the $(uv)_0$ equations;

(c) adding $uuv$ terms (ditto) to the $(vv)$ equations.

Thus, the equations of $C$ itself can be written schematically as
\begin{equation}\label{C-deform-eq}
\begin{array}{l}
(uu)_0+\pi^d(v)+\pi^{2d}(uu)=0,\\
(uv)_0+\pi^d(uuu)+\pi^{2d}(uv)=0,\\
v_iv_j-p_{ij}(u)+\pi^d(uuv)+\pi^{2d}(vv)+\pi^{2d}(uuuu).
\end{array}
\end{equation}
Let us modify the equations of $C$ by dividing the $v$ 
variables by $\pi^d$ and clearing denominators as necessary: 
\begin{equation}\label{C'-deform-eq}
\begin{array}{l}
(uu)_0+(v)+\pi^{2d}(uu)=0,\\
(uv)_0+\pi^{2d}(uuu)+\pi^{2d}(uv)=0,\\
v_iv_j-\pi^{2d}p_{ij}(u)+\pi^{2d}(uuv)+\pi^{2d}(vv)+\pi^{4d}(uuuu).
\end{array}
\end{equation}
We call the obtained family $C'$.
We immediately see that $C'\mod \fm^{2d}$ is a ribbon in canonical form (see Lemma \ref{ribbon-eq-syz-lem}).
Note that the normal vector to the hyperelliptic locus associated with $C$
determines the terms $(v)\mod\fm$ in the equations \eqref{C-deform-eq}
(see Lemma \ref{normal-ribbon-equation-lem}(ii)),
so $C'\mod {\frak m}$ depends only on this normal vector.

Now Lemma \ref{normal-ribbon-equation-lem}(i) implies that
$C'\mod \fm^{2d+1}$ is not a ribbon, and that
the normal vector to the ribbon locus coming from $C'\mod \fm^{2d+1}$
depends only on $p_{ij}(u)$.
\ed

\begin{rmk}  Without the constraint on the residue characteristic, one can 
still perform the above calculation modulo $\pi^d$, and find that 
$C'_{R/{\frak m}^d}$ is indeed a ribbon (the ribbon corresponding in the 
usual way to the normal vector to the hyperelliptic locus).  For a 
general hyperelliptic curve in characteristic $2$, it is no longer true 
that $C'_{R/{\frak m}^{2d}}$ is a ribbon.  Moreover, even when $2$ is 
invertible, neither the ribbon $C'_{R/{\frak m}^{2d}}$ nor the 
corresponding normal vector are uniquely determined by 
$C_{R/{\frak m}^{2d}}$; only their images modulo ${\frak m}^d$ are so determined.
\end{rmk}

Next, we consider the inverse procedure of going from a family with ribbon special fiber
to a family with hyperelliptic special fiber. 

\begin{prop}\label{ribbon-to-hyperell-prop}
Suppose that $C/R$ is a point (over a dvr $R$ with residue 
field not of characteristic $2$) of the canonical Hilbert scheme $H_g(\P^{g-1})$.

\noindent
(i) Assume that for some integer $d>0$, $C_{R/{\frak m}^{2d}}$ is a ribbon but 
$C_{R/{\frak m}^{2d+1}}$ is not.  Then the special fiber of the normalization of $C$ 
along the special fiber is the double cover of 
$\P^1$ obtained by adjoining a square root of the associated section of 
$\sO_{\P^1}(2g+2)$. Furthermore, the corresponding normal vector to the hyperelliptic locus is determined
by $C\mod{\frak m}$.

\noindent
(ii) Assume that for some odd integer $n>0$, $C_{R/{\frak m}^n}$ is a ribbon but
$C_{R/{\frak m}^{n+1}}$. Then the assertion of (i) holds after replacing $C/R$ with the induced family over $R'$,
where $\Spec(R')\to \Spec(R)$ the ramified double cover corresponding
to taking the square root of a uniformizer. 
\end{prop}

\Pf . (i) The process is just the inverse of that of Proposition \ref{hyperell-to-ribbon-prop}.
Starting with $C$, a (nonhyperelliptic) ribbon to order precisely $2d$, with 
nonhyperelliptic special fiber, we can lift it to a subscheme of $X_g$ with equations of the form
\eqref{C'-deform-eq}. Then rescaling the variables $(v_i)$ 
by $\pi^d$ gives a curve with equations \eqref{C-deform-eq}.
which is hyperelliptic to order precisely $d$, since
by Lemma \ref{normal-ribbon-equation-lem}(ii), the nontrivial linear terms $(v)$
determine the corresponding normal vector to the hyperelliptic locus.
Also, by Lemma \ref{normal-ribbon-equation-lem}(i), 
this hyperelliptic curve modulo ${\frak m}^d$ is the double cover corresponding to the section 
of $\sO_{\P^1}(2g+2)$ coming from the normal vector to the ribbon 
locus corresponding to $C$. 

Note that if we had instead rescaled the variables $(v_i)$ by $\pi^l$ for some $1\le l<d$, then 
the resulting special fiber would have been $y^2=0$, and the reduced 
special fiber would have been singular in the new family $C_l$.  Moreover, 
$C_l$ is the blowup in the reduced special fiber of $C_{l-1}$, and thus $C_d$ 
is the result of a sequence of codimension 1 blowups, so has the same 
(partial) normalization as $C$.  But the special fiber of $C_d$ has isolated 
singularities, so $C_d$ is the desired partial normalization.


\noindent
(ii) A ramified quadratic 
base change doubles $n$ without affecting smoothness of the generic fiber and without changing the normal vector
to the ribbon locus.
\ed




\begin{rmk} 
    A similar argument shows directly (again with residue 
characteristic not $2$) that if $C$ is a ribbon to odd order $2d+1$, then its 
relative normalization has nonreduced special fiber.  Indeed, it is 
enough to show that after rescaling by $\pi^d$, the singular locus is 
$0$-dimensional, which by semicontinuity reduces to the case when
$C_{R/{\frak m}^{2d+1}}$ is the hyperelliptic ribbon.  Rescaling gives a 
curve which is a ribbon to order $1$ and again semicontinuity allows us to 
assume the curve is hyperelliptic to order $2$.  But then the curve has 
the form $y^2 + \pi h$, which is singular precisely on the 0-dimensional 
subscheme where $h$ vanishes.
\end{rmk}

\begin{rmk}
      Suppose we are given a point of the canonical Hilbert scheme over an 
equicharacteristic dvr such that the special fiber is a ribbon and the 
general fiber is smooth.  Then the following is true in characteristic 
not 2: the singular subscheme of the total space is the subscheme of the 
reduced special fiber cut out by the image in $\Gamma(\sO_{\P^1}(2g+2)) $
of the tangent vector induced by the family.  Note that since the family 
is generically smooth, the singular subscheme of the total space is 
contained in the singular subscheme of the special fiber, which since 
the characteristic is not 2 is just the reduced special fiber.  (The 
ribbon locally looks like $\Spec(k[x,\epsilon]/\epsilon^2)$.)  In general 
(for families which are not generically smooth or in characteristic 2), 
the subscheme cut out by the element of $\Gamma(\sO_{\P^1}(2g+2))$ is 
merely the intersection with the reduced special fiber of the singular 
subscheme of the total space.
\end{rmk}

We will use the following criterion for proving regularity of a rational map.

\begin{lem}\label{Smyth-lem} (\cite[Lem.\ 4.2]{Smyth})
Let $f:X\dashrightarrow Y$ be a birational map of algebraic spaces with $X$
normal. Assume that $X$ and $Y$ are open subsets of proper algebraic spaces. Then
to check that $f$ is well defined at a given point $x\in X$ we have to 
consider various maps
from a dvr to $X$ sending the closed point to $x$ and the generic point to 
the open locus where $f$ is defined.
If their composition with $f$ always has the same limit in $Y$ then $f$ is 
well defined at $x$.
\end{lem}

\Pf . Smyth proves this assuming the spaces are proper.
Let $\ov{X}$ and $\ov{Y}$ are proper spaces containing $X$ and $Y$ as dense open subsets.
It remains to replace $\ov{X}$ by its normalization and apply the result for proper spaces.
\ed

\noindent
{\it Proof of Theorem \ref{two-blow-ups-thm}}.
Let $U\sub\Bl_{\ov{R_g}}H_g(\P^{g-1})$ be the complement to the union of the preimage of $\ov{R}_g\setminus R_g$,
the strict transform of the locus of singular curves
and $\ov{\DD}$.  Note that $U$ is the disjoint union of the locus of smooth curves in $H_g(\P^{g-1})$ and
of $\P N_{R_g}H_g(\P^{g-1})\setminus \DD$.

Proposition \ref{ribbon-to-hyperell-prop} together with Lemma
\ref{Smyth-lem} imply that there is a regular morphism
$$U\to [\Bl_{\wt{\Hyp}_g}\wt{\MM}_g]^{\coarse}.$$
Conversely, using Proposition \ref{hyperell-to-ribbon-prop} and Lemma \ref{Smyth-lem}
we get a regular morphism 
$$[\Bl_{\wt{\Hyp}_g}\wt{\MM}_g]^{\coarse}\to \Bl_{\ov{R_g}}H_g(\P^{g-1}).$$
To show that the image is in $U$, it suffices to show that for a family of smooth curves $C$ over a dvr,
hyperelliptic up to order $d$, the normal vector to the ribbon locus of the corresponding family in $H_g(\P^{g-1})$
does not lie in the discriminant locus. 
But this follows from the construction of Proposition
\ref{hyperell-to-ribbon-prop} and from Lemma \ref{normal-ribbon-equation-lem}.
\ed

\begin{rmk}
More generally, we conjecture that over $\Z$, the blown up $\PGL_g$-bundle
remains an open subscheme of the blow up of the Hilbert scheme in the
ribbon locus (defined as the closure from $\Z[1/2]$); the odd behavior in
characteristic $2$ reflects the fact that the fiber over $2$ of the
ribbon subscheme is nonreduced.
\end{rmk}

\end{document}